\documentclass[opre,nonblindrev,copyedit]{informs2_TR}

\usepackage{appendix}
\usepackage{endnotes}

\usepackage{natbib}
 \bibpunct[, ]{(}{)}{,}{a}{}{,}%
 \def\bibfont{\small}%
\usepackage[bookmarks=true,pageanchor,colorlinks,linkcolor=blue,anchorcolor=blue,citecolor=black,urlcolor=blue]{hyperref}

\newcommand{\bx}{\mathbf{x}}

\usepackage{bookmark}
\usepackage{multirow}
\usepackage{color}

\long\def\old#1{}
\long\def\com#1{{\color{green}[#1]}}
\def\ds#1{\displaystyle{#1}}
\def\mg#1{{\color{magenta}#1}}

\TheoremsNumberedThrough     
\ECRepeatTheorems

\EquationsNumberedThrough    

\begin{document}


 \RUNAUTHOR{Tsitsiklis and Xu}

\RUNTITLE{Efficiency Loss in a Cournot Oligopoly with Convex Market
Demand}

\TITLE{Efficiency Loss in a Cournot Oligopoly with Convex Market
Demand
}

\ARTICLEAUTHORS{%
\AUTHOR{John N. Tsitsiklis and Yunjian Xu}
\AFF{Laboratory or Information and Decision Systems, MIT, Cambridge, MA, 02139, \EMAIL\{jnt@mit.edu, yunjian@mit.edu\}} 
} 

\ABSTRACT{We consider a Cournot oligopoly {model} where multiple
suppliers (oligopolists) compete by choosing quantities. We compare
the social welfare achieved at  a Cournot equilibrium to the maximum
possible, for the case where the {inverse} market demand function is
convex. We establish a lower bound on the efficiency of Cournot
equilibria in terms of a scalar parameter {derived} from the inverse
demand function, namely, the ratio of the slope of the inverse
demand function at the Cournot equilibrium to the {average slope of
the inverse demand function between} the Cournot equilibrium and {a}
social optimum.  {Also,  for the} case of a single, monopolistic,
profit maximizing supplier, or of multiple suppliers who collude to
maximize their total profit, we establish a {similar but tighter}
lower bound on the efficiency of the resulting output.
 Our results provide
nontrivial quantitative bounds on the loss of social welfare for
several convex inverse demand functions that {appear} in the
economics literature.}
%


\KEYWORDS{Cournot oligopoly, price of anarchy} \HISTORY{This paper
was first submitted in October 2011 (revised March 2012).}

\maketitle

%

\section{Introduction}
We consider a Cournot oligopoly model where multiple suppliers
(oligopolists) compete by choosing quantities, with a focus on the
case where the inverse market demand function is convex. Our
objectives are {to compare the optimal social welfare to: (i) the
social welfare at a Cournot equilibrium and (ii) the} social welfare
achieved when the suppliers collude to maximize the total profit,
or, equivalently, when there is a single supplier.

\subsection{Background}

In {a book} 
on oligopoly theory (see Chapter 2.4 of
\cite{F83}), Friedman {raises} two questions on the relation between
Cournot equilibria and competitive equilibria. First, ``is the
Cournot equilibrium close, in some reasonable sense, to the
competitive equilibrium?'' {Furthermore,} ``will the two equilibria
coincide as the number of firms goes to infinity?'' The answer to
the second question seems to be positive, {in general. Indeed,} the
efficiency properties of Cournot equilibria in \old{[large]}
economies and markets {with a large number of suppliers and/or
consumers} have been much explored. For {the case of} a large number
of suppliers,
it is {known} that  every Cournot equilibrium is 
approximately a socially optimal competitive equilibrium
\citep{GV72, NS78, N80}. {Furthermore,} \cite{U85} derives necessary
and/or sufficient conditions on the {relative numbers} of consumers
and suppliers for the efficiency loss associated with every Cournot
equilibrium to approach zero, as the number of suppliers increases
to infinity.

In more recent work, attention has turned to the efficiency of
Cournot equilibria in {settings that involve} an arbitrary (possibly
small) {number} of
 suppliers or consumers.  \cite{AR03} quantify the efficiency loss
{in} Cournot oligopoly models with concave demand functions.
However, most of their results focus on the relation {between}
consumer surplus, producer surplus, and the aggregate social welfare
achieved at a Cournot equilibrium, rather than {on} the relation
between the  social welfare  achieved at a Cournot equilibrium and
the optimal social welfare.

{The concept of efficiency loss is intimately related to the concept
of
 ``price of anarchy,''} {advanced} by
Koutsoupias and Papadimitriou in a {seminal} paper \citep{KP99};
{it} provides a natural measure {of} the difference between a
Cournot equilibrium and a socially optimal competitive equilibrium.
{In the spirit of \cite{KP99}, {we define} the efficiency of a
Cournot equilibrium}  as the ratio of its aggregate social welfare
to the optimal social welfare.
Recent works have reported various efficiency bounds for Cournot
oligopoly with affine demand functions. \cite{KP08} \old{Kluberg and
Perakis} compare the social welfare and the aggregate profit earned
by {the} suppliers under Cournot competition to the corresponding
maximum possible, {for the case where} suppliers produce multiple
differentiated products and {demand is an affine function of the
price.} {Closer to the present paper,}  \cite{JT05} {establish a
$2/3$ lower bound on} the efficiency of a Cournot equilibrium, when
the inverse demand function is affine. They also show that the
$2/3$ lower bound applies to a monopoly model with general concave
demand.

 The efficiency loss in a Cournot oligopoly
 with some specific forms of convex inverse demand functions {has received some recent attention. \cite{C08} studies the special case
of convex inverse demand functions of the form}
$p(q)=\alpha-\beta q^\gamma$,  analyzes the efficiency
loss at a Cournot equilibrium and shows that when $\gamma>0$,
 the worst case efficiency loss occurs when an efficient
{supplier} has to share the market with infinitely many inefficient
suppliers. {\cite{GY05} consider a class of inverse demand functions
that solve a certain differential equation (for example, constant
elasticity inverse demand functions belong to this class), and
establish} efficiency lower bounds that depend on equilibrium market
shares, the market demand, and the number of suppliers.

In this paper, we 
study the {efficiency
loss} in a 
Cournot oligopoly model with
general convex demand functions\footnote{Since a demand
function is generally nonincreasing, the convexity of a demand
function implies that the corresponding inverse demand function is
also convex.}. Convex demand functions, such as the negative
exponential and the constant elasticity demand curves, have been
widely used in oligopoly analysis and marketing research
\citep{BP83,FW09,T99}. {In general,} a Cournot equilibrium {need}
not exist when the inverse demand function is convex. However, it is
well known that a Cournot equilibrium {will exist} if the inverse
demand function is ``not too convex'' {(e.g., if the inverse demand
function is concave),} {in which case the quantities supplied by
different suppliers are} strategic substitutes \citep{BGK85,BP03}.
Existence results for Cournot oligopolies {for the case of}
strategic substitutes can be found in \cite{N85,GS91,S91}, and
\cite{LS00}.
 Note {however, that the} strategic substitutes {condition} is
not necessary for the existence of Cournot equilibria. {For example,
using the theory of} ordinally supermodular games,  \cite{A96} shows
that the log-concavity of inverse demand functions guarantees the
existence of a Cournot equilibrium. {In this paper, we will not
address the case of concave inverse demand functions, which appears
to be qualitatively different, as will be illustrated by an example
in Section \ref{sec:concave}}.

\subsection{Our contribution}


For Cournot oligopolies with convex {and nonincreasing} demand
functions, we establish a lower bound {of the form} $f(c/d)$ on the
efficiency achieved at {a} Cournot equilibrium. {Here, $f$ is a
function given in closed form; $c$ is the absolute value of the
slope of the inverse demand function at the Cournot equilibrium; and
$d$  is the absolute value of the slope of the
 line that
agrees with the inverse demand function at the Cournot equilibrium
and at a socially optimal point.} {For convex and nonincreasing
inverse demand functions,  we} have $c \ge d$; {for affine inverse
demand functions, we have $c/d=1$. In {the latter} case, our
efficiency bound is $f(1)=2/3$, which is consistent with the bound
derived in \cite{JT05}. More generally, the ratio $c/d$ can be
viewed as a  measure of nonlinearity of the inverse demand function.
As the ratio $c/d$ goes to infinity, the lower bound converges to
zero and arbitrarily high efficiency losses are possible. The
usefulness of this result lies in that it allows us to lower bound
the efficiency of Cournot equilibria for a large class of Cournot
oligopoly models {in terms of qualitative properties of the inverse
demand function,} without having to restrict to the special case of
affine demand functions, {and without having to calculate the
equilibrium and the social optimum.}

An interesting special case of our model arises when $N=1$, in which
case we are dealing with a single, monopolistic, supplier. The
previous lower bounds for Cournot equilibria continue to hold.
However, by using the additional assumption that $N=1$, we can hope
to obtain sharper (i.e., larger) lower bounds in terms of the same
scalar parameter $c/d$. Let us also note that the case  $N=1$ also
covers a setting where there are multiple suppliers who choose to
collude and coordinate production so as to maximize their total
profit.

\old{Through the methodology used in our efficiency loss analysis,
we then construct a theoretical framework to compare the social
welfare, consumer surplus, and supplier profit achieved at a Cournot
Equilibrium, a Social Optimum (SO) where the social welfare is
maximized, and a Monopoly Output (MO) where the aggregate profit of
suppliers is maximized. The constructed framework is shown in Table
\ref{table:main}. We derive a lower bound on the ratio of the
aggregate profit earned by all suppliers at a CE to the maximum
possible aggregate profit, that is, the profit that would have been
achieved if the suppliers were to collude at a MO. The lower bound
depends on the number of suppliers, as well as on the same scalar
parameter $c/d$.}

\subsection{Outline of the paper}

The rest of the paper is organized as follows. In the next section,
we formulate the model and review {available} {results on} the
existence of Cournot {equilibria.} In Section \ref{sec:property}, we
{provide} some mathematical preliminaries on Cournot equilibria
{that} will be useful later, {including the fact that efficiency
lower bounds can be obtained by {restricting to} linear cost
functions. In Section \ref{sec:affine}, we {consider affine inverse
demand functions and derive a refined lower bound on the efficiency
of Cournot equilibria that depends on a small amount of ex post
information. We also show this bound to be tight.} In Section
\ref{sec:convex}, we consider a more general model, {involving}
convex inverse demand functions. {We show that for convex inverse
demand functions, and for the purpose of studying the worst case
efficiency loss, it suffices to restrict to a special class of
piecewise linear inverse demand functions. \old{[Based on this fact,
we can relate  the worst case efficiency loss to a
finite-dimensional optimization problem, which we solve in closed
form.]} This leads to} the main result of this paper, a lower bound
on the efficiency of Cournot equilibria (Theorem \ref{thm:convex}).
Based on this theorem, in Section \ref{sec:app} we derive a number
of corollaries that provide efficiency lower bounds that can be
calculated without {detailed} information on {these equilibria,}
{and apply these results to various commonly encountered convex
inverse demand functions.}
 In Section
\ref{sec:mono}, we establish a lower bound on the efficiency of
monopoly outputs (Theorem \ref{thm:mono}), and show by example
that the social welfare  at a monopoly output can be higher than
that achieved at a Cournot equilibrium (Example
\ref{example:mono}).
 Finally, in Section \ref{sec:con}, {we make some brief concluding remarks.}

\bookmarksetup{startatroot}

\section{Formulation and Background} \label{sec:model}

{In this section, we define the Cournot competition model that we
study in this paper. We also review {available} results on the
existence of Cournot equilibria.}

{We} consider a market for a single homogeneous good with inverse
demand function $p: [0, \infty) \to {[0,\infty)}$ and $N$ suppliers.
Supplier $n \in \{1,2,\ldots,N\}$ has a cost function $C_n: [0,
\infty) \to [0, \infty)$. Each supplier $n$ {chooses} a nonnegative
real number $x_n$, which is the amount of the good {to be supplied
{by her.} The \textbf{strategy profile} ${\bx}=(x_1,x_2,\ldots,
x_N)$ results in a total supply denoted by \(X = \sum\nolimits_{n =
1}^N {{x_n}} \), and a corresponding market price $p(X)$.} The
payoff {to} supplier $n$ is
$$
{\pi_n}({x_n},{{\bx}_{ - n}}) = {x_n}p({X})
 - {C_n}({x_n}),
$$
{where we have used the standard notation ${\bx}_{-n}$ to indicate
the vector $\bf x$ with the component $x_n$ omitted.} A strategy
profile $\bx=(x_1,x_2,\ldots, x_N)$ is a Cournot (or Nash)
equilibrium if
\[
{\pi_n}({x_n},{{\bx}_{ - n}}) \ge {\pi_n}(x,{{\bx}_{ -
n}}),\qquad\forall\ x \ge 0,\ \ \forall\ n \in \{ 1,2, \ldots ,N\}.
\]

{In the sequel, we denote by $f'$ and $f''$ the first and second,
respectively, derivatives of a scalar function $f$, if they exist.}
{For a function defined on a domain $[0,Q]$, the derivatives at the
endpoints $0$ and $Q$ are defined as left and right derivatives,
respectively. For points in the interior of the domain, and if the
derivative is not guaranteed to exist,} {we use the notation
$\partial_+f$ and $\partial_-f$ to denote the right and left,
respectively, derivatives of $f$; these are guaranteed to exist for
convex or concave functions~$f$.}

\subsection{Existence results}\label{se:exist}

Some results on the existence of Cournot equilibrium are provided by
\cite{SY77}, but require the concavity of the inverse demand
function. \old{ that a Cournot equilibrium exists when: (i)
$p(\cdot)$ is nonincreasing, concave, and twice continuously
differentiable, on the interval where its value is positive; and,
(ii) the cost functions $C_n(\cdot)$ are nondecreasing, twice
continuously differentiable, and convex. } \cite{M64} provides an
existence result under minimal assumptions on the inverse demand
function $p(\cdot)$, but only for the special case where all
suppliers have the same \old{(convex)} cost function $C(\cdot)$ ---
a rather restrictive assumption. {The most relevant result for our
purposes is provided by \cite{N85} who does not require the
suppliers to be identical or the inverse demand functions to be
concave.}

\old{{[Check that the assumptions of these papers are quoted
correctly.]}}

\begin{proposition} \label{Prop:Novshek85} \citep{N85}
Suppose that {the following conditions hold:}
\begin{enumerate}
\item[(a)] The inverse
demand function $p(\cdot)$ is continuous.

\item[(b)] There exists a real number $Q>0$ such that {$p(q) = 0$ for $q\geq Q$. Furthermore,} $p(\cdot)$ is twice continuously differentiable
and strictly decreasing on $[0, Q)$.

\item[(c)] For every $q \in {[}0,Q)$, {we have} $p'(q)+qp''(q) \le 0$.

\item[(d)] {The cost functions  $C_n(\cdot)$, $n=1,2,\ldots,N$, are}
nondecreasing and lower-semi-continuous.
\end{enumerate}
Then, there exists a Cournot equilibrium.
\end{proposition}

{If the inverse demand function $p(\cdot)$ is convex}, the condition
(c) in the preceding proposition {implies that}
\[\frac{{{\partial ^2}{\pi _n}}}{{\partial {x_n}\partial {x_m}}}(X)
\le 0,\qquad\forall\ m \ne n, \ \  \forall\ X \in (0,Q),\] i.e.,
{that} the quantities supplied by different suppliers are} strategic
substitutes. {We finally note that \citet{A96} proves existence of a
Cournot equilibrium in a setting where the strategic substitutes
condition does not hold. Instead, this reference assumes that the
 inverse
demand function $p(\cdot)$  is strictly decreasing and log-concave.}

\
\old{\begin{proposition} \rm{\textbf{(Amir (1996))}}
\label{Prop:Amir96} Suppose that
\begin{enumerate}
\item The inverse
demand function $p(\cdot)$  is strictly decreasing and
log-concave\footnote{Since $p(\cdot)$ is strictly decreasing, we let
$\mathcal {P}$ be the convex subset of $[0,\infty)$ such that
$p(q)>0$ for every $q$ in $\mathcal {P}$. We say $p(\cdot)$ is
log-concave if the function, $\log (p): \mathcal {P} \to
[0,\infty)$, is concave. }.

\item For every $n$, $C_n(\cdot)$ is strictly increasing and
left-continuous.

\item There exists some real number $Q>0$ such that $p(Q) = 0$.
\end{enumerate}
Then there exists a Cournot equilibrium.
\end{proposition}

\begin{example}
The following two convex inverse demand functions do not satisfy the
conditions required in Proposition \ref{Prop:Novshek85}, but the
conditions in the previous proposition:
\[
p(q) = \left\{ \begin{array}{l}
 {(Q - q)^2},\;\;\;\;\;{\rm{if}}\;\;\;0 \le q \le Q, \\
 0,\;\;\;\;\;\;\;\;\;\;\;\;\;\;\;\;{\rm{if}}\;\;\;\; q> Q,\; \\
 \end{array} \right.
 \]
and
\[p(q) = {e^{ - aq}},\;\;\;0 \le q,\]
where $a$ is some positive constant.
\end{example}
}

\section{Preliminaries on Cournot Equilibria}\label{sec:property}
In this section, we introduce several {main} assumptions {that we
will be working with, {and some definitions.}
 In Section \ref{sec:property1},  we present
conditions for a nonnegative  vector to be a social optimum or a
Cournot equilibrium. Then, in Section
\ref{sec:property2}, we 
define the efficiency of a Cournot equilibrium. In Sections
\ref{sec:linear-worst} and  \ref{sec:property3}, we derive some
properties {of} Cournot equilibria {that will be useful later, but
which may also be of some independent interest.} For example, we
show that the worst case efficiency occurs when the cost functions
are linear. \old{[{These} results will be {used} in Sections
\ref{sec:affine} and \ref{sec:convex}.]} The proofs of all
propositions in this section (except for Proposition
\ref{Prop:trivial}) are given in Appendix \hyperlink{page.29}{A}.

\begin{assumption}\label{A:cost}
 For any $n$, 
the cost function $C_n:[0,\infty) \to [0,\infty)$ is  convex,
continuous, and nondecreasing on $[0,\infty)$, and continuously
differentiable on $(0,\infty)$. Furthermore, $C_n(0)=0$.
\end{assumption}

\begin{assumption}\label{A:demand}
The inverse demand function $p: [0,\infty) \to [0,\infty)$ is
continuous, nonnegative, and nonincreasing, with $p(0)>0$. Its right
derivative at $0$ exists and {at} every $q>0$, its left and right
derivatives {also} exist.
\end{assumption}
Note that we do not {yet} assume that the inverse demand function is
convex. {The reason is that}  {some of} the results to be derived in
this section  are valid even in the absence of such a convexity
assumption. Note also that some parts of our assumptions are
redundant, but are included for easy reference. For example, if
$C_n(\cdot)$ is convex and nonnegative, with $C_n(0)=0$, then it is
automatically continuous and nondecreasing.

\begin{definition}\label{Def:optimal}
 The \textbf{optimal social welfare} 
is the optimal {objective} value  in the following optimization
problem,
\begin{equation}\label{equa:optimal}
\displaystyle
\begin{array}{l}
 {\rm{maximize}}\quad\ds{\int_0^{{X} } {p(q)\,dq}  - \sum_{n = 1}^N {C_n({x_n})}}  \\
 {\rm{subject}}\;{\rm{to}}\quad{x_n} \ge 0,\qquad n = 1,2, \ldots ,N,
 \end{array}
 \end{equation}
{where $X=\sum_{n=1}^N x_n$.}
\end{definition}

In the above definition, $\int_0^{{X} }
{p(q)\,dq}$ is the aggregate consumer surplus 
and $\sum_{n = 1}^N {C_n({x_n})}$ is the total cost {of} {the}
suppliers. The objective function in (\ref{equa:optimal}) is a
measure {of} the social welfare across the entire economy of
consumers and suppliers, the same measure as {the one} used in
\cite{U85} and \cite{AR03}.

{For a model with a nonincreasing {continuous} inverse demand
function and {continuous} convex cost functions, the following
assumption guarantees the existence of an optimal solution to
\eqref{equa:optimal},} {because it essentially restricts the
optimization to the compact set of vectors $\bx$ for which $x_n\leq
R$, for all $n$.}

\begin{assumption}\label{A:optimal}
{There exists some $R>0$ such that $p(R) \le \min_n \{ C'_n(0) \}$.}
\end{assumption}

\subsection{Optimality and equilibrium conditions}\label{sec:property1}

We observe that under Assumption \ref{A:cost} and \ref{A:demand},
the objective function in (\ref{equa:optimal}) is concave. Hence, we
have the following \emph{necessary and sufficient} conditions for a
{vector} $\bx^S$ to achieve the optimal social welfare:
\begin{equation}\label{equa:optimality}
\left\{ \begin{array}{l}
\displaystyle C'_n(x_n^S) = p\left({X^S} \right),\;\;\;\;{\rm{if}}\;x_n^S > 0, \\[5pt]
\displaystyle C'_n(0) \ge p\left({X^S} \right),\;\;\;\;\;\;\; {\rm{if}}\;x_n^S = 0, \\
 \end{array} \right.
 \end{equation}
{where $X^S=\sum_{n=1}^N x_n^S$.}


The social optimization problem \eqref{equa:optimal} may admit
multiple optimal solutions. However, as we now show, they must all
result in the same price. We note that the differentiability of the
cost functions is crucial for this result to hold.

{\begin{proposition} \label{Prop:sameprice} Suppose that Assumptions
\ref{A:cost} and \ref{A:demand} hold. All optimal solutions to
(\ref{equa:optimal}) result in the same price.
\end{proposition}}


\old{
 \begin{definition}\label{Def:com}
A nonnegative vector $(x_1,\ldots,x_N)$ is a \textbf{competitive
equilibrium}, if
\[{x_n} \in \arg {\max _{x \ge 0}}\left\{ {xp\left( X \right) - {C_n}(x)} \right\},\;\;\;\;n = 1, \ldots ,N,\]
where \(X = \sum\nolimits_{n = 1}^N {{x_n}} \).
 \end{definition}

\begin{proposition} \label{Prop:competitive}
Suppose that Assumptions \ref{A:cost} and \ref{A:demand} hold. Then,
 a
nonnegative vector $\textbf{x}$ is a competitive equilibrium if and
only if it is a solution to the optimization problem in
(\ref{equa:optimal}).

\end{proposition}

}

{There are similar {equilibrium} conditions for a strategy profile
$\bx$. In particular, {under Assumptions \ref{A:cost} and
\ref{A:demand},} if $\textbf{x}$ is a Cournot equilibrium, then}
\begin{align}\label{equa:nece1}
& C'_n(x_n^{}) \le p\left( X \right) + {x_n} \cdot \partial_- p \left( X \right),\;\;\;\;{\rm{if}}\;x_n > 0, \\
& C'_n(x_n^{}) \ge p\left( X \right) + {x_n} \cdot \partial_+
p\left( X \right), \label{eq:necsec}
 \end{align}
where {again} $X=\sum\nolimits_{n = 1}^N {{x_n}} $. {Note, however,
that in the absence of further assumptions, the payoff of supplier
$n$ need not be a {concave} function of $x_n$ and these conditions
are, in general, not sufficient.}

{We will say that a nonnegative vector $\bx$ is a \textbf{Cournot
candidate} if it satisfies the necessary conditions
(\ref{equa:nece1})-\eqref{eq:necsec}. Note that for a given model,
the set of Cournot equilibria is a subset of the set of Cournot
candidates. Most of the results obtained in this section, including
the efficiency lower bound in Proposition \ref{Prop:linear}, apply
to all Cournot candidates.}

For convex inverse demand functions, the necessary conditions
 (\ref{equa:nece1})-\eqref{eq:necsec} can be further
refined.  

\begin{proposition}\label{Prop:derivative}
 Suppose that Assumptions \ref{A:cost}
 and \ref{A:demand} hold, {and} {that} the inverse demand function {$p(\cdot)$} is convex. 
 {If $\bx$ is} a Cournot  {candidate} with $X=\sum_{n
= 1}^N {{x_n}}>0$,  {then   $p(\cdot)$ must be differentiable at
$X$,} i.e.,
\[\partial_- p \left( X \right)=
\partial_+ p\left( X \right).
\]
\end{proposition}

Because of the above proposition, {when {Assumptions \ref{A:cost}
and \ref{A:demand} hold and the} inverse demand function is convex,}
we have the following necessary (and, by definition, sufficient)
conditions for a nonzero vector $\mathbf{x}$ to be a Cournot
{candidate:}

\begin{equation}\label{equa:nece}
\left\{ \begin{array}{l}
 \displaystyle C'_n(x_n^{}) = p\left( X \right) + {x_n}p'(X),\;\;\;\;{\rm{if}}\;x_n^{} > 0,
 \\[5pt]
\displaystyle  C'_n(0) \ge p\left( X \right) + {x_n}p'(X),\;\;\;\;\; \; \;{\rm{if}}\;x_n^{} = 0. \\
 \end{array} \right.
 \end{equation}


\subsection{Efficiency of Cournot equilibria}\label{sec:property2}

 As shown in \cite{F77}, {if
$p(0) > \min_n\{C'_n(0)\}$, then the aggregate supply at a Cournot
equilibrium is positive}; see Proposition \ref{Prop:trivial} below
for a slight generalization. {If on the other hand $p(0)\leq
\min_n\{C'_n(0)\}$, then the model is uninteresting, because no
supplier has an incentive to produce and the optimal social welfare
is zero. This motivates the assumption that follows.}

\begin{assumption}\label{A:p0}
The price at zero supply is larger than the minimum marginal cost of
the suppliers, i.e.,
$$
p(0) > \min_n\{C'_n(0)\}.
$$
\end{assumption}

\begin{proposition}\label{Prop:trivial}
Suppose that Assumptions \ref{A:cost},   \ref{A:demand}, and
\ref{A:p0}   hold. If $\bx$ is a Cournot {candidate}, then $X>0$.
\end{proposition}

\proof{Proof} Suppose that $p(0) > \min_n\{C'_n(0)\}$. {Then, the
vector $\mathbf{x}=(0,\ldots,0)$ violates condition
(\ref{eq:necsec}), and} cannot be a Cournot candidate.\Halmos
\endproof}

{Under Assumption \ref{A:p0}, at least one supplier has an incentive
to choose a positive quantity, which leads us to the  next result.}

\begin{proposition}\label{Prop:positive}
Suppose that Assumptions \ref{A:cost}-\ref{A:p0} hold. Then, the
social welfare achieved at a Cournot {candidate}, {as well as} the
optimal
 {social welfare [cf.\ (\ref{equa:optimal})]}, are positive.
\end{proposition}

{We now define the efficiency of a Cournot equilibrium as the ratio
of the social welfare that it achieves to the optimal social
welfare. It is actually convenient to define the efficiency of a
general vector $\bx$, not necessarily a Cournot equilibrium. }

\begin{definition}\label{Def:efficiency}
Suppose that Assumptions \ref{A:cost}-\ref{A:p0} hold. The
\textbf{efficiency} of a nonnegative vector
$\textbf{x}=(x_1,\ldots,x_N)$ is defined as
\begin{equation} \label{eq:gamma}
\gamma (\textbf{x}) = \dfrac{\ds{{\int_0^{{X} } {p(q)\,dq}  -
\sum_{n = 1}^N {C_n(x_n)} }}}{\ds{{\int_0^{{X^S} } {p(q)\,dq}  -
\sum_{n = 1}^N {C_n({x^S_n})}} }},
 \end{equation}
where ${\bx^S=}(x_1^S,\ldots, x_N^S)$ is an optimal solution of the
optimization problem in (\ref{equa:optimal}) {and
$X^S=\sum\nolimits_{n = 1}^N {{x^S_n}}$.}
\end{definition}

{We note that $\gamma(\bx)$ is well defined:
 because of Assumption \ref{A:p0} {and Proposition \ref{Prop:positive},}
  the denominator {on the right-hand side of \eqref{eq:gamma}} is guaranteed to be positive. Furthermore, {even}
if there are multiple socially optimal solutions $\bx^S$, the value
of the denominator is the same for all such $\bx^S$.} {Note that
$\gamma(\bx)\leq 1$ for every nonnegative vector $\bx$.}
 {Furthermore, if $\bx$ is a Cournot {candidate}, then
$\gamma(\bx)>0$, by Proposition \ref{Prop:positive}.} 

\subsection{{Restricting to linear cost functions}}\label{sec:linear-worst} {In this section, we show that in order to study the worst-case efficiency of Cournot equilibria, it suffices to consider linear cost functions.
We first provide a lower bound on $\gamma(\bx)$ and then proceed to
interpret  it. }

\begin{proposition}\label{Prop:linear}
 Suppose that Assumptions \ref{A:cost}-\ref{A:p0} hold {and that $p(\cdot)$ is convex.}
Let $\mathbf{x}$ be a {Cournot candidate which is not socially
optimal,} {and let $\alpha_n=C'_n(x_n)$.} {Consider a modified model
in which we replace the cost function of each supplier $n$ by a new
function $\overline C_n(\cdot)$, defined by}
$$
\overline C_n(x)=\alpha_n x,\qquad \forall\ x \ge 0.
$$
Then, {for the modified model,  Assumptions \ref{A:cost}-\ref{A:p0}
still hold, the vector $\mathbf{x}$ is a Cournot candidate, and its
efficiency, denoted by $\overline \gamma(\mathbf{x})$, satisfies}
{$0<  \overline \gamma(\mathbf{x}) \leq \gamma(\mathbf{x})$.}
\end{proposition}

If $\textbf{x}$ is a Cournot equilibrium, then it satisfies Eqs.\
(\ref{equa:nece1})-(\ref{eq:necsec}), and therefore is a Cournot
candidate. Hence, Proposition \ref{Prop:linear} applies to all
Cournot equilibria {{that} are not socially optimal. {We} note that
if {a} Cournot candidate $\mathbf{x}$ is socially optimal {for} the
original {model, then the optimal social welfare in the modified
model could be zero, in which case $\gamma(\bx)=1$, but $\overline
\gamma(\mathbf{x})$ is undefined; see the example that follows.}

\begin{example}\label{example:linear}
{\rm Consider a  model involving two suppliers ($N=2$). The cost
function of supplier $n$ is $C_n(x)=x^2$, for $n=1,2$.
 The inverse demand function is constant, {with}  $p(q)=1$ for
 any $q \ge 0$. It is not hard to see that the vector $(1/2,1/2)$
 is a Cournot candidate, which is also socially optimal.
In the modified {model}, we have $\overline C_n(x)=x$, for $n=1,2$.
The optimal social welfare achieved in the modified model is zero.}
\Halmos
\end{example}

{Note that even if $\bx$ is a Cournot equilibrium in the original
{model}, it {need} not be a Cournot equilibrium in the modified
{model}  with linear cost functions, as {illustrated by our next}
example. On the other hand, Proposition \ref{Prop:linear} asserts
that a Cournot candidate {in} the original {model} remains a Cournot
candidate in the modified {model}. Hence, to {lower} bound the
efficiency of a Cournot equilibrium in the original {model, it
suffices to lower bound the efficiency} achieved at {a} worst
Cournot candidate for a modified model.} {Accordingly, and for the
purpose of deriving lower bounds, we can (and will)  restrict to the
case of linear cost functions, and study the worst case efficiency
over all Cournot candidates.}

\begin{example}\label{example:candidate}
{\rm{Consider a  model involving only one supplier ($N=1$). The cost
function of the supplier is $C_{{1}}(x)=x^2$.
 The inverse demand function is given by}}
$$
p(q) = \left\{ \begin{array}{ll}
  -q+4,& {\rm if}\ \ 0 \le q \le 4/3, \\
 \max\{0,- {\dfrac{1}{5}}(q-4/3)+8/3\}, \ \ \ \ \ \ & {\rm if}\ \ 4/3 < q,
 \end{array} \right.
$$
\rm{{which is convex and satisfies Assumption \ref{A:demand}. It can
be verified that $x_1=1$ maximizes the supplier's profit and thus is
a Cournot equilibrium in the original {model}. In the modified
{model}, $\overline C_1(\cdot)$ is linear with a slope of $2$; the
supplier can maximize its profit at $x_1=7/3$. Therefore, in the
modified {model}, $x_1=1$ remains a Cournot candidate, but not a
Cournot equilibrium.  }} \Halmos
\end{example}


\old{ {{Note that} a vector $\bx$ is a {\it Cournot candidate} (for
a given model) if it satisfies {the necessary conditions
(\ref{equa:nece1})-\eqref{eq:necsec}.} Then, Proposition
\ref{Prop:linear} can be interpreted as follows. For a given inverse
demand function $p(\cdot)$, the worst case efficiency over all
Cournot {candidates which are  not socially optimal} (where the
worst case is taken over all possible cost functions $C_n(\cdot)$
and all corresponding equilibria) can be lower bounded by the worst
case efficiency over all linear cost functions and associated
Cournot candidates. Accordingly, and for the purpose of deriving
lower bounds, we can (and will)  restrict to the case of linear cost
functions, and study the worst case efficiency over all Cournot
candidates.}}

\subsection{Other properties of Cournot {candidates}}\label{sec:property3}

{In this subsection, we collect}\old{We finally note} a few {useful
and} intuitive  properties of Cournot {candidates}.  We \old{also}
show that at a Cournot candidate there are two possibilities: either
{$p(X)>p(X^S)$ and $X<X^S$,  or $p(X)=p(X^S)$ (Proposition
\ref{Prop:less}); {in the latter} case, under the additional
assumption that $p(\cdot)$ is convex, {a} Cournot {candidate} is
socially optimal (Proposition \ref{Prop:equal}). In {either case,}
imperfect competition can never result in a price that is less than
the socially optimal price.}

\begin{proposition}\label{Prop:less}
Suppose that Assumptions \ref{A:cost}-\ref{A:p0} hold. Let ${\bx}$
and ${\bx}^S$ be a Cournot {candidate} and {an optimal} solution to
(\ref{equa:optimal}), respectively. If $p(X)\ne p(X^S)$, then
{$p(X)>p(X^S)$ and} \(
 X <  X^S
\).
\end{proposition}

{For the case where $p(X)=p(X^S)$, Proposition \ref{Prop:less} does
not provide any comparison between $X$ and $X^S$. While one usually
has $X<X^S$ (imperfect competition results in lower quantities), it
is also possible that $X > X^S$, as in the following example.}

\begin{example}\label{example:less}
{\rm Consider a {model involving} two suppliers {($N=2$).} The cost
function of each supplier is linear, with slope {equal to} $1$. The
inverse demand function is {convex,} of the form
$$
p(q) = \left\{ \begin{array}{l}
  2-q,\;\;\;\; \;\;\;\; {\rm if}\ \ 0 \le q \le 1, \\
 1, \;\;\;\;\;\;\;\;\;\;\;\;\;\;\;\; {\rm if}\ \ 1 < q .
 \end{array} \right.
$$
It is not hard to see that any nonnegative vector $\textbf{x}^S$
 that satisfies $x^S_1+x^S_2 \ge 1$ is socially optimal;
$x^S_1=x^S_2=1/2$ is one such vector. On the other hand, it can be
verified that $x_1=x_2=1$ is a Cournot equilibrium. Hence, in this
example, $2=X > X^S =1$. }  \Halmos
\end{example}

\old{
\begin{remark}\label{re:less}
{\rm Since the inverse demand function is nonincreasing, Proposition
\ref{Prop:less} shows that imperfect competition  raises the
equilibrium price above marginal cost, i.e., that the price at a
Cournot equilibrium cannot be less than the price at a socially
optimal point. }
\end{remark}
}

\begin{proposition}\label{Prop:equal}
 Suppose that Assumptions \ref{A:cost}-\ref{A:p0} hold and {that} the inverse demand function is convex.
 Let ${\bx}$ and ${\bx}^S$ be a Cournot
{candidate} and {an optimal} solution to (\ref{equa:optimal}),
respectively. If $p(X)=p(X^S)$, then $p'(X)=0$ and $\gamma(\bx)=1$.
\end{proposition}

{Proposition \ref{Prop:sameprice} {shows} that all social optima
lead to a unique ``socially optimal'' price. {Combining with}
Proposition \ref{Prop:equal}, we conclude that if $p(\cdot)$ is
convex, a Cournot candidate is socially optimal if and only if it
{results in} the socially optimal price.}

\subsection{Concave inverse demand functions}\label{sec:concave}
{In this {section, we argue} that the case of concave inverse demand
functions is fundamentally different. For this reason, the study of
the concave case would require a very different line of analysis,
and is not considered further in this paper.}

{According to} Proposition \ref{Prop:equal}, {if the inverse demand
function is convex and} if the price at a Cournot equilibrium equals
{the price} at a socially optimal point, then the Cournot
equilibrium is socially optimal. {For nonconvex inverse demand
functions, this is not necessarily true: a socially optimal
 price can be associated with a socially suboptimal
Cournot equilibrium,} as demonstrated by the following example.

\begin{example}\label{example:concave1}   {\rm
Consider {a model involving two suppliers ($N = 2$), with}
$C_1(x)=x$  and $C_2(x)=x^2$.
 The inverse demand function is  {concave}  {on the interval where it is
 positive},  {of the form}
$$
p(q) = \left\{ \begin{array}{ll}
  1,& {\rm if}\ \ 0 \le q \le 1, \\
 \max\{0, -M(q-1)+1 \}, \ \ \ \ \ & {\rm if}\ \ 1 < q,
 \end{array} \right.
$$
where $M>2$. It is not hard to see that the vector $(0.5,0.5)$
satisfies the optimality conditions in (\ref{equa:optimality}), and
is therefore socially optimal. We now argue that $(1/M,1-1/M)$ is a
Cournot equilibrium. Given the action {$x_2=1/M$} of supplier $2$,
any action on the interval $[0,1-1/M]$ is a best response for
supplier $1$. Given the action {$x_1=1-(1/M)$} of supplier $1$, a
simple calculation shows that
$$
\argmax_{x \in [0, \infty)} \left\{ x \cdot p(x+1-1/M) -x^2 \right\}
=1/M.
$$
Hence, $(1/M,1-1/M)$ is a Cournot equilibrium. Note that $X=X^S=1$.
However, the optimal social welfare is $0.25$, while the social
welfare achieved at the Cournot equilibrium is $1/M-1/M^2$. By
considering arbitrarily large $M$, the corresponding efficiency can
be made arbitrarily small.} \Halmos
\end{example}

{The preceding example shows that arbitrarily high efficiency losses
are possible, even if $X=X^S$.} {The possibility of inefficient
allocations even when the price is the correct one opens up the
possibility of substantial inefficiencies that are hard to bound. }

\section{Affine Inverse Demand Functions}\label{sec:affine}

{We now turn our attention to the special case of}  affine inverse
demand functions. {It is already known from \cite{JT05} that 2/3 is
a tight lower bound on the efficiency of Cournot equilibria. In this
section, we refine this result by providing a tighter lower bound,
based on a small amount of ex post information about a Cournot
equilibrium.}

{Throughout this section, we assume an inverse demand function of
the form}
\begin{equation}\label{equa:affine}
p(q) = \left\{ \begin{array}{lll}
  b- aq,&& {\rm if}\ \ 0 \le q \le {b}/{a}, \\
 0, &&{\rm if}\ \ {b}/{a} <q , \\
 \end{array} \right.
 \end{equation}
where $a$ and $b$ are positive constants.\footnote{Note that the
model considered here is slightly different from that in
\cite{JT05}. In {that} work, the inverse demand function {is
literally affine and} approaches minus infinity as the total supply
increases to infinity. However, as remarked in that paper (p.\ 20),
this difference does not affect the results.} Under the assumption
of convex costs (Assumption \ref{A:cost}), a Cournot equilibrium is
guaranteed to exist, by Proposition \ref{Prop:Novshek85}.
\old{Furthermore, by Proposition \ref{Prop:condition}, a nonnegative
vector $\bx$ {such that $X \le b/a$} is a Cournot equilibrium if and
only if it is a Cournot candidate, i.e., if for all $n$,
\begin{equation}\label{eq:aff-opt}
\begin{array}{lcl}
 {C'_n}({x_n}) = p\left( {{X} } \right) - a{x_n},&\qquad&{\rm{if}}\;{x_n} > 0,\\
 {C'_n}(0) \ge p\left( {{X} } \right),&& {\rm if\; x_n=0.}
 \end{array}
 \end{equation}}

{The main result of this section follows.}

\begin{theorem}\label{thm:linear}
 Suppose that Assumption \ref{A:cost} holds {(convex cost functions),}
 and that the
inverse demand function is {affine, of the form
\eqref{equa:affine}.\footnote{{Note that Assumptions \ref{A:demand}
and \ref{A:optimal} hold automatically.}} Suppose also that} $b >
\min_n\{C'_n(0)\}$ (Assumption \ref{A:p0}). {Let $\bx$ be a Cournot
equilibrium, and} let $\alpha_n=C'_n(x_n)$. {Let also}
$$
\beta = \dfrac{a {X}} {b-\min_n\{\alpha_n\}},
$$
{If $X>b/a$, then $\mathbf{x}$ is socially optimal. Otherwise:}
\begin{itemize}
\item[(a)]
{We have}
 $ 1/2 \le \beta < 1$.
\item[(b)] The efficiency of $\bx$
 satisfies,
$$
\gamma(\bx) \ge g(\beta) = 3 \beta^2 - 4 \beta +2.
$$
\item[(c)]
{The bound in part (b)} is tight. That is, for every
$\beta\in[{1}/{2},1)$ {and every $\epsilon>0$,} there exists a
{model} {with a} Cournot equilibrium {whose efficiency} is {no more
than} $g(\beta){+\epsilon}$.
\item[(d)] The function  $g(\beta)$ is minimized at $ \beta =2/3$ {and} the worst case efficiency is $2/3$.
\end{itemize}

\end{theorem}

Theorem \ref{thm:linear} is proved in Appendix \hyperlink{B1}{B.1}.
\old{because it involves the same line of argument as the proof of
Theorem \ref{thm:convex} in the next section.} {The  lower bound
$g(\beta)$ is illustrated in Fig.\ 1.} {Consider a Cournot
equilibrium such that $X \le b/a$.} {For the special case where all
the cost functions are linear, of the form $C_n(x_n)=\alpha_n$,
Theorem \ref{thm:linear} has an interesting interpretation. {We
first note} that a {socially} optimal solution is obtained when the
price $b-aq$ equals the marginal cost of a ``best'' supplier, namely
$\min_n \alpha _n$. In particular, $X^S=(b-\min_n \{\alpha_n\})/a$,
and $\beta=X/X^S$. Since $p'(X) =-a <0$, Proposition
\ref{Prop:equal} {implies} that $p(X) \ne p(X^S)$, and Proposition
\ref{Prop:less} implies that $\beta<1$. Theorem \ref{thm:linear}
{further} states that $\beta\geq 1/2$. i.e., that the total supply
at a Cournot equilibrium is at least half of the socially optimal
supply. Clearly, if $\beta$ is close to $1$ we expect \old{that} the
efficiency loss due to {the difference $X^S-X$} {to} be small.
However, efficiency losses may also arise if the total supply at a
Cournot equilibrium is not provided by the most efficient suppliers.
(As shown in Example \ref{example:concave1}, in the nonconvex case
this effect can be substantial.) Our result shows that, for the
convex case, $\beta$ can be used to lower bound the total efficiency
loss due to this second factor as well; when $\beta$ is close to
$1$, the efficiency indeed remains close to $1$. {(This is in sharp
contrast to the nonconvex case where we can have $X=X^S$ but large
efficiency losses.)} Somewhat surprisingly, the worst case
efficiency also tends to be somewhat better for low $\beta$, that
is, when $\beta$ approaches $1/2$, as compared to intermediate
values ($\beta\approx 2/3$).}

 \begin{figure}\label{Fig:linear}
    \includegraphics[width=10.5cm]{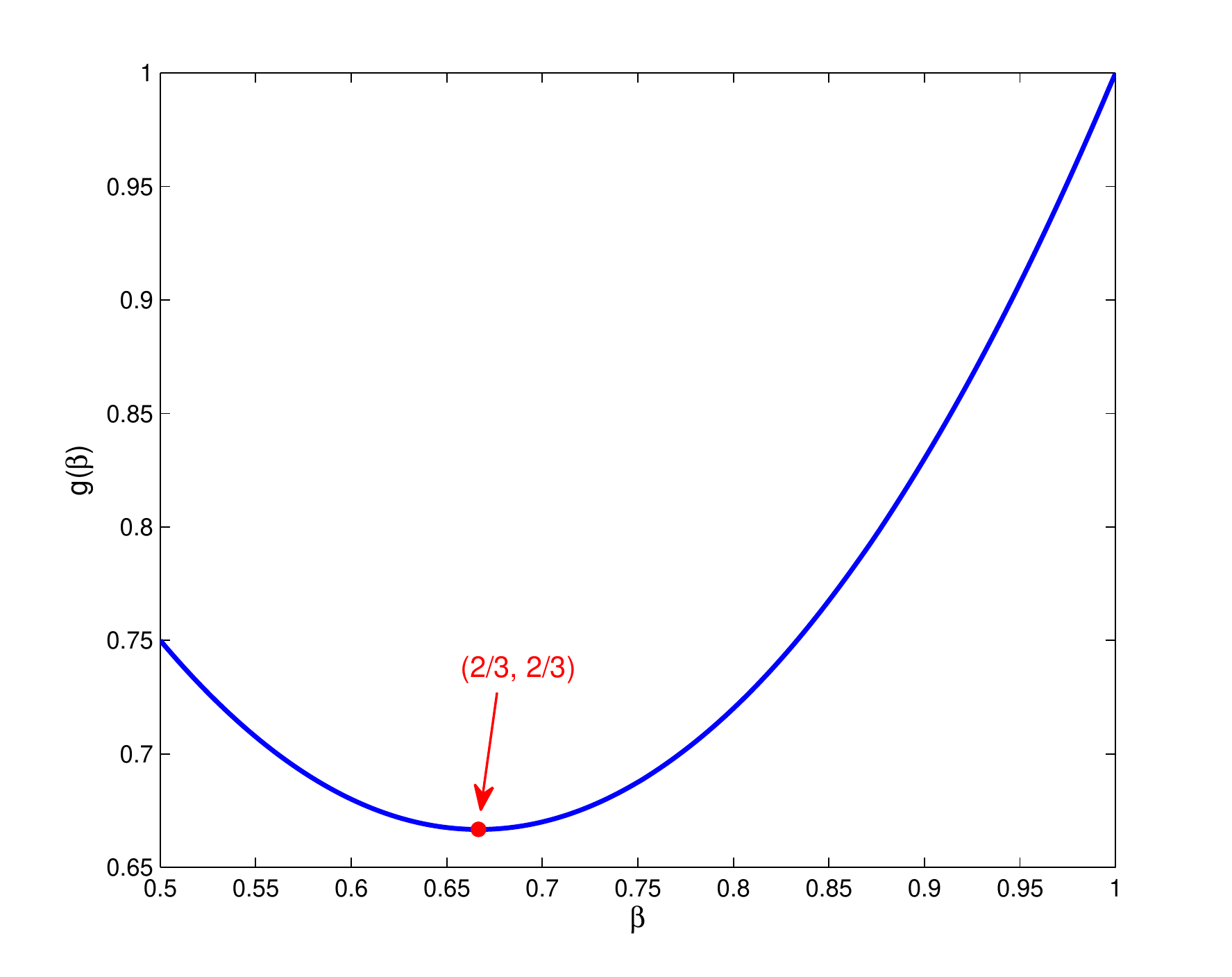}
    \centering
    \caption{A tight lower bound {on} the efficiency of  Cournot equilibria {for the case of} affine inverse demand functions.}
  \end{figure}

\section{Convex Inverse Demand
Functions}\label{sec:convex}

In this section we study the efficiency of Cournot equilibria {under
more general assumptions.} Instead of restricting the inverse demand
function to be affine,  we will {only} assume that it is convex. A
Cournot equilibrium {need} not exist in general, but it does exist
under some conditions (cf.\ {Section \ref{se:exist}}). Our results
apply whenever a Cournot equilibrium happens to exist.

We first show that a lower bound on the efficiency of a Cournot
equilibrium can be established by calculating its efficiency in
another {model} with a piecewise linear
 inverse demand function. Then, in Theorem \ref{thm:convex}, we establish a lower bound
on the efficiency of Cournot equilibria, as a function of the ratio
of the slope of the inverse demand function at the Cournot
equilibrium to the {average slope of the inverse demand function
between} the Cournot equilibrium and a socially optimal point. Then,
in Section \ref{sec:app}, we will apply Theorem \ref{thm:convex} to
{specific} convex inverse demand functions. \old{ and establish
nontrivial lower bounds on the efficiency of Cournot equilibria.}
{Recall our definition of a Cournot candidate as a vector $\bx$ that
satisfies the necessary conditions
(\ref{equa:nece1})-\eqref{eq:necsec}.}

\begin{proposition}\label{Prop:piecewise_linear}
Suppose that Assumptions \ref{A:cost}-\ref{A:p0} hold, and {that}
the inverse demand function is convex. Let ${\bx}$ and ${\bx}^S$ be
a Cournot
 {candidate} and an  {optimal} solution to (\ref{equa:optimal}), respectively.
Assume that $p(X) \ne p(X^S)$ and let\footnote{According to
Proposition \ref{Prop:derivative}, $p'(X)$ must exist.} $c =|p' (X)
|$. Consider  {a modified model} in which we  {replace} the inverse
demand function by a new function $p^0(\cdot)$, defined by

\begin{equation}\label{equa:piecewise_linear}
p^0(q) = \left\{ \begin{array}{lcl}
  - c(q - X) + p(X), && {\rm if}\ \ \  0 \le q \le X, \\[5pt]
 \max \left\{ {0, \dfrac{{p({X^S})   - p(X )}}{{{X^S} - X }}(q - X) + p(X)\;} \right\}, &&  {\rm if}\ \ \ X<q.
 \end{array} \right.
\end{equation}
Then, for the modified model, with inverse demand function
$p^0(\cdot)$, the vector $\bx^S$ remains socially optimal, and the
efficiency of $\bx$, denoted by $\gamma^0(\bx)$, satisfies
\[
\gamma^0 ({\bx}) \le \gamma ({\bx}).
\]
\end{proposition}

 \begin{figure}\label{Fig:piecewise_linear}
    \includegraphics[width=10.5cm]{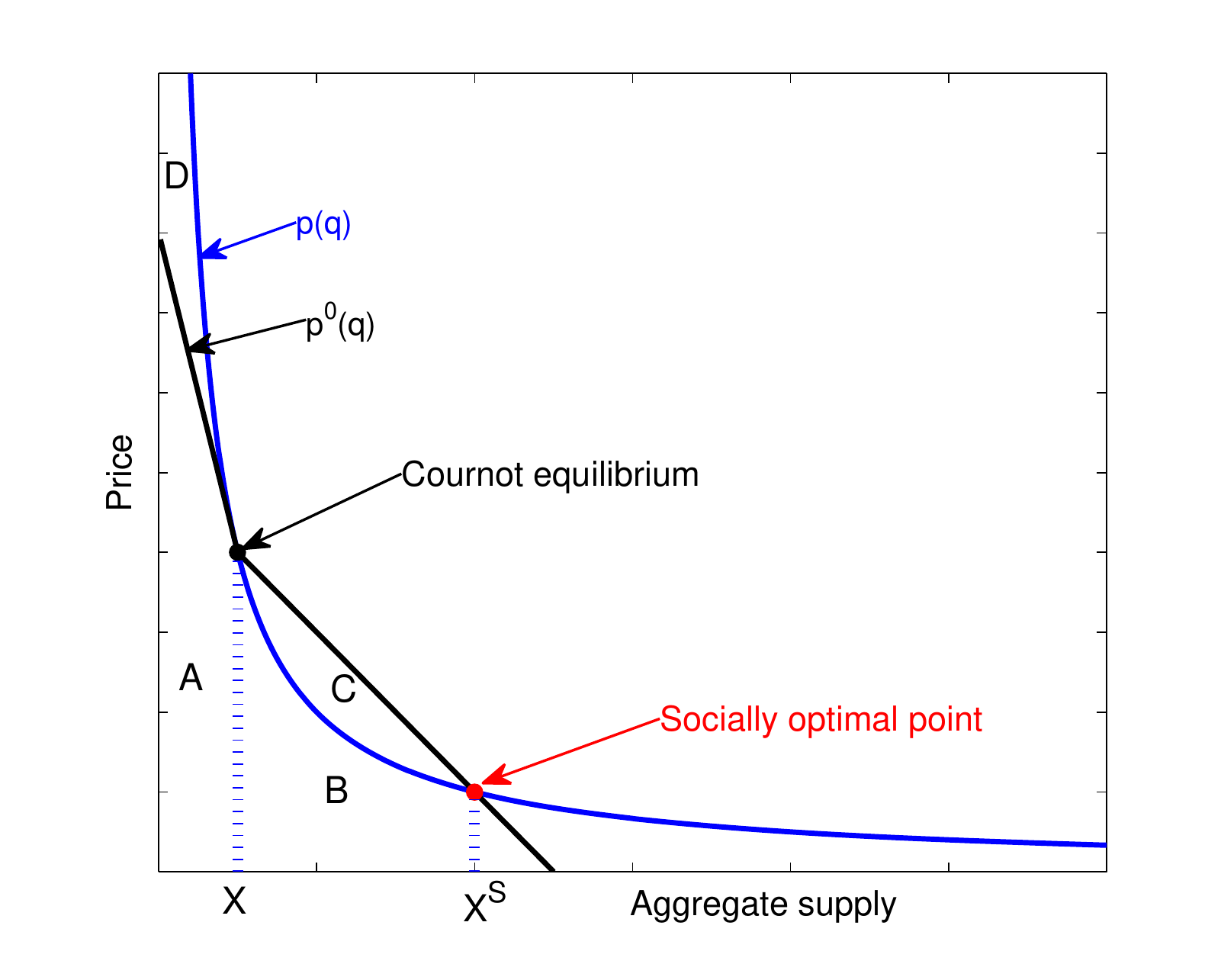}
    \centering
    \caption{The efficiency of a Cournot equilibrium cannot increase if we replace the inverse demand {function} by the piecewise linear function $p^0(\cdot)$.
    The function $p^0(\cdot)$ is tangent to the inverse demand {function}  $p(\cdot)$ at the equilibrium point, and connects the Cournot equilibrium point with the socially optimal
    point.}
  \end{figure}

\proof{Proof} Since $p(X) \ne p(X^S)$, Proposition \ref{Prop:less}
{implies} that $X<X^S$, so that $p^0(\cdot)$ is well defined.
\old{Furthermore, since $\bx$ is a Cournot candidate for the
original model, we have
$$
C'_n(x_n){\leq p\left( X \right)+x_n p'(X)} =p(X) - {x_n}c =
p^0\left( X \right) + {x_n} {\cdot \partial_- p^0}\left( X \right),
\;\;\;\;{\rm{if}}\;x_n
> 0.
$$
This proves that $\bx$ {satisfies the necessary conditions
\eqref{equa:nece1} } for the modified model.} \old{Furthermore,}
{Since}  the {necessary and sufficient} optimality conditions in
(\ref{equa:optimality}) only involve the value of the inverse demand
function at $X^S$, {which has been unchanged,} the
 vector $\textbf{x}^S$ remains socially optimal {for the modified model.}

{Let}
\[
{{{A}}} = \int_0^X {{p^0}(q)\, dq}, \; {{{B}}} = \int_X^{X^S}
{{p}(q)\, dq},\; {{{C}}} = \int_X^{X^S} { (p^0(q) - p (q))\, dq},\;
{{{D}}} = \int_0^{X} {  (p(q) - p^0 (q))\, dq}.
\]
{See Fig.\ 2 for an illustration of $p(\cdot)$ and a graphical
interpretation of $A$, $B$, $C$, $D$.} Note that since $p(\cdot)$ is
convex, we have $\emph{C} \ge 0$ and $D\geq 0$. The efficiency of
${\bx}$ in the original
 {model} with inverse demand function $p(\cdot)$, is
$$
0 <\gamma({\bx})=\dfrac{{{{A}}}+{{{D}}}-\sum\nolimits_{n = 1}^N
{{C_n}({x_n})} }{{{{A}}}+{{{B}}}+{{{D}}}-\sum\nolimits_{n = 1}^N
{{C_n}({x^S_n})}} \le 1,
$$
{where the first inequality is true because the social welfare
achieved at any Cournot candidate is positive  (Proposition
\ref{Prop:positive}).}
 The efficiency of $\textbf{x}$ in the {modified model} is
$$
\gamma^0({\bx})=\dfrac{{{{A}}}-\sum\nolimits_{n = 1}^N
{{C_n}({x_n})} }{{{{A}}}+{{{B}}}+{{{C}}}-\sum\nolimits_{n = 1}^N
{{C_n}({x^S_n})}}.
$$
{Note that the denominators in the above formulas for $\gamma(\bx)$
and $\gamma^0(\bx)$ are {all} positive, by Proposition~\ref{Prop:positive}.}

{If $ {{A}}-\sum\nolimits_{n = 1}^N {{C_n}({x_n})} \le 0$, then
$\gamma^0(\bx)\le0$ and the result is clearly true. We can therefore
assume that} $ {{A}}-\sum\nolimits_{n = 1}^N {{C_n}({x_n})} > 0. $
We then have
\begin{eqnarray*}
 0 &<& \gamma^0(\bx) = \dfrac{ A-\sum\limits_{n = 1}^N {{C_n}({x_n})}
}{A+B+C-\sum\limits_{n = 1}^N {{C_n}({x^S_n})}} \le
\dfrac{A+D-\sum\limits_{n = 1}^N {{C_n}({x_n})}
}{A+B+C+D-\sum\limits_{n = 1}^N {{C_n}({x^S_n})}}\\  &\le&
\dfrac{A+D-\sum\limits_{n = 1}^N {{C_n}({x_n})}
}{A+B+D-\sum\limits_{n = 1}^N {{C_n}({x^S_n})}} =\gamma(\bx)\le 1,
\end{eqnarray*}
which proves the desired result. \Halmos
\endproof}

 {Note that unless $p(\cdot)$ happens to be linear on
the interval $[X,X^S]$,  the function $p^0(\cdot)$ is not
differentiable at $X$ and, according to} Proposition
\ref{Prop:derivative}, $\textbf{x}$ cannot be a Cournot {candidate}
\old{for the Cournot competition} {for the modified model}\old{with
inverse demand function $p^0(\cdot)$}. {Nevertheless,} $p^0(\cdot)$
can {still} be used to derive a lower bound on the efficiency of
Cournot {candidates} {{in} the original model,} {as will be seen in
the proof of Theorem \ref{thm:convex}, which is our main result.}

\old{We are now ready for the main result of this section.}

\begin{figure}\label{Fig:convex}
    \includegraphics[width=10cm]{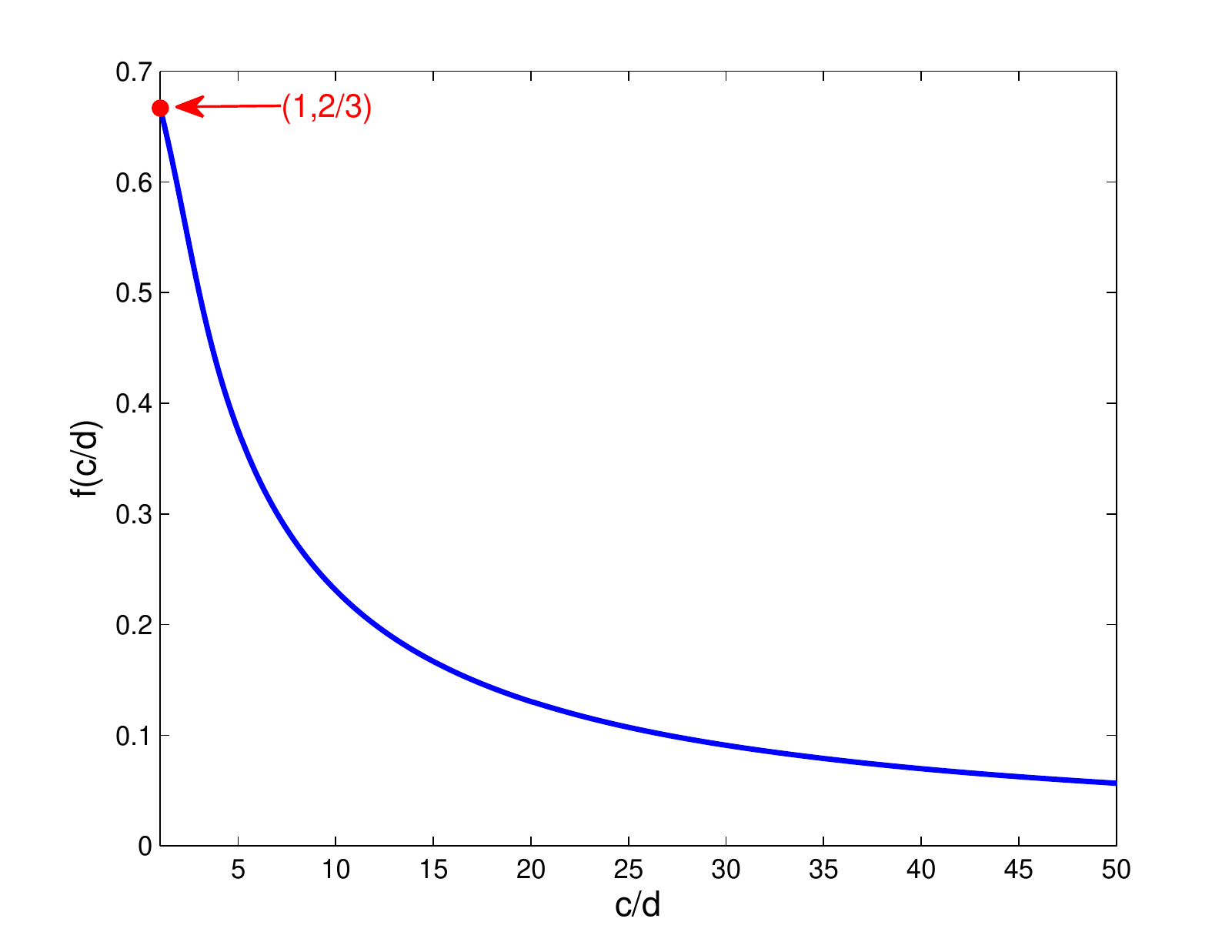}
    \centering
    \caption{Plot of the lower bound {on} the efficiency of a Cournot equilibrium in a Cournot oligopoly with convex inverse demand functions,
    {as a function of the ratio $c/d$ (cf. Theorem \ref{thm:convex}).}}
  \end{figure}

\begin{theorem}\label{thm:convex}
 Suppose that Assumptions \ref{A:cost}-\ref{A:p0} hold, and {that} the inverse demand function is convex. Let ${\bx}$ and ${\bx}^S$ be a
Cournot {candidate} and a solution to (\ref{equa:optimal}),
respectively. {Then, the following hold.}
\begin{enumerate}

\item[(a)] If $p(X)=p(X^S)$, then $\gamma({\bx})=1 $.

\item[(b)] {If} $p(X)\ne p(X^S)$, let
$c=|p'(X)|$, $d=|(p(X^S)-p(X))/(X^S-X)|$, and $\overline c= c/d$. We
have {$\overline c \ge 1$ and }
\begin{equation}\label{equa:thm_convex1}
1> \gamma({\bx}) \ge f(\overline c)=\dfrac{\phi^2 +2}{\phi^2+ 2
\phi+ \overline c} ,
\end{equation}
where
$$
{\phi = \max \left\{ {\frac{{2-\overline c   + \sqrt {{{\overline c
}^2} - 4\overline c  + 12} }}{2},1} \right\}.}
$$
\end{enumerate}
\end{theorem}

\begin{remark}{\rm
We do not know whether the lower bound in Theorem \ref{thm:convex}
is tight. The difficulty in proving tightness is due to the fact
that the vector $\bx$ need not be a Cournot equilibrium in the
modified model.}
\end{remark}

We provide the proof of Theorem \ref{thm:convex} in Appendix
\hyperlink{B2}{B.2}. {Fig.\ 3 shows a plot of} the lower bound
{$f(\overline c)$} {on} the efficiency of Cournot equilibria, as a
function of $\overline c=c/d$. If $p(\cdot)$ is affine, then
$\overline c=c/d=1$. From (\ref{equa:thm_convex1}), {it can be
verified that} $f({1})=2/3$, which agrees with the lower bound {in
\cite{JT05} for the} affine case. {We note that} the lower bound
$f(\overline c)$ is monotonically decreasing in $\overline c$, over
the domain $[1,\infty)$. When $\overline c \in [1,3)$, ${\phi}$ is
{at least} $1$, and monotonically decreasing in $\overline c$. When
$\overline c \ge 3$, ${\phi}=1$.

\section{Applications}\label{sec:app}

\noindent {For a given} inverse demand function $p(\cdot)$, the
lower bound derived in Theorem \ref{thm:convex} requires  {some}
knowledge on the Cournot candidate and the social optimum, {namely,}
the aggregate supplies $X$ and $X^S$. {Even so, for a large class}
of inverse demand functions, we {can} apply \old{the results in}
Theorem \ref{thm:convex} to establish lower bound{s} on the
efficiency of Cournot equilibria {that do not require knowledge of
$X$ and $X^S$.} With additional information on {the} suppliers' cost
functions, the \old{derived} lower bounds can be further refined. At
the end of this section, we apply {our} results to calculate
nontrivial {quantitative} efficiency bounds for {various} convex
inverse demand functions {that have been considered in the}
economics literature.

\begin{corollary}\label{Coro:mu}
Suppose that Assumptions \ref{A:cost}-\ref{A:p0} hold and that the
inverse demand function is convex. Suppose also that $p(Q)=0$ for
some $Q>0$, and that the  ratio, $\mu  =
\partial_+p(0)/
\partial_-p(Q) $, is finite. Then, the efficiency of a Cournot candidate
is at least $f(\mu)$.
\end{corollary}

The proof of Corollary \ref{Coro:mu} can be found in Appendix
\hyperlink{B3}{B.3}. For convex inverse demand functions, e.g.,
{for} negative exponential demand, {with}
$$
p(q)= \max \{0, \alpha-\beta \log q\}, \qquad  0< \alpha, \qquad 0<
\beta, \qquad 0 \le q ,
$$
 Corollary \ref{Coro:mu} {does not} apply,
because the left derivative of $p(\cdot)$ at $0$ is infinite. This
motivates us to refine the lower bound  in Corollary \ref{Coro:mu}.
By using a small amount of {additional} information on {the} cost
functions, we {can} derive an upper bound on the total supply at a
social optimum, as well as a lower bound on the {total} supply at a
Cournot equilibrium, to strengthen Corollary \ref{Coro:mu}.

\begin{corollary}\label{Coro:st_lower_bound}

Suppose that Assumptions \ref{A:cost}-\ref{A:p0} hold and {that}
$p(\cdot)$ is convex. Let\footnote{Under Assumption \ref{A:optimal},
the existence of the real numbers defined in \eqref{equa:t} is
guaranteed.} \old{ $t$ be a real number  such that}
\begin{equation}\label{equa:t}
s=\inf \{q \mid p(q) = \min_n C'_n(0) \},\;\;\;t = \inf \left\{ {q \
\big| \ \min_{n} C'_n(q) \ge p(q) + q
\partial_+p(q)} \right\}.
\end{equation}
If $\partial_-p(s)<0$, then the efficiency of a Cournot candidate is
{at least} $f\left(\partial_+p(t)/
\partial_-p(s)\right)$.

\end{corollary}

\begin{remark} {\rm
If there exists a ``best'' supplier $n$ such that $ C'_n(x) \le
C'_m(x)$, for any other supplier $m$ and any $x >0$,  then the
parameters $s$ and $t$ depend only on $p(\cdot)$ and $C'_n(\cdot)$.}
\end{remark}

Corollary \ref{Coro:st_lower_bound} is proved in Appendix
\hyperlink{B4}{B.4}. We now apply Corollary
\ref{Coro:st_lower_bound} to three examples.

\begin{example}\label{example:eq6_BP83}
\rm{Suppose that Assumptions \ref{A:cost}, \ref{A:optimal}, and
\ref{A:p0} hold, and {that} there is a best supplier, whose cost
function is linear with a slope $c \ge 0$. Consider inverse demand
functions { of the form ({cf.\ Eq. (6) in} \cite{BP83})}
\begin{equation}\label{equa:eq6_BP83}
p(q)= \max \{0, \alpha-\beta \log q\},\qquad 0<q,
\end{equation}
{where $\alpha$ and $\beta$ are positive constants.} Note that
Corollary \ref{Coro:mu} does not apply, because the left derivative
of $p(\cdot)$ at $0$ is infinite.\footnote{{In fact, $p(0)$ is
undefined. This turns out to not be an issue: for a small enough
$\epsilon>0$, we can guarantee that no supplier chooses a quantity
below $\epsilon$. Furthermore, $\lim_{\epsilon\downarrow
0}\int_0^\epsilon p(q)\,dq =0$. For this reason, the details of the
inverse demand function in the vicinity of zero are immaterial as
far as the chosen quantities or the resulting social welfare are
concerned. }}
 Since
$$p'(q)+qp''(q) = \dfrac{-\beta}{q} + \dfrac{q\beta}{q^2} =0, \qquad \forall q \in (0,\exp(\alpha/\beta)), $$ Proposition
\ref{Prop:Novshek85} {implies} that there exists a Cournot
equilibrium. Through a simple calculation we obtain
$$
s=\exp\left(\dfrac{\alpha-c}{\beta}\right),\qquad t=\exp\left(
{\dfrac {\alpha-\beta-c}{\beta}}   \right).
$$
From Corollary \ref{Coro:st_lower_bound} we obtain that for every
Cournot equilibrium $\textbf{x}$,
\begin{equation}\label{equa:example61}
\gamma (\textbf{x}) \ge f\left(
\dfrac{\exp\left((\alpha-c)/\beta\right)}{\exp\left(
(\alpha-\beta-c)/\beta \right)} \right) =f\left( \exp\left( 1
\right) \right )\ge 0.5237.
\end{equation}

Now we argue that the efficiency lower bound \eqref{equa:example61}
holds even without the assumption that there is a best supplier
associated with a linear cost function.
 From
Proposition \ref{Prop:linear}, the efficiency of any Cournot
equilibrium $\mathbf{x}$ will not increase if the cost function of
each supplier $n$ is replace{d} by
$$
\overline C_n(x) = C'_n(x_n) x,\qquad\forall x \ge 0.
$$
Let $c=\min_n \{C'_n(x_n) \}$. Since the efficiency lower bound in
(\ref{equa:example61}) holds for the modified {model} with linear
cost functions, it applies {whenever the} inverse demand function
 {is of} the form
(\ref{equa:eq6_BP83}). } \Halmos
\end{example}

\begin{example}\label{example:eq5_BP83}
\rm{Suppose that Assumption{s}  \ref{A:cost}, \ref{A:optimal}, and
\ref{A:p0} hold, and {that} there is a best supplier, {whose} cost
function is linear with a slope $c \ge 0$. Consider inverse demand
functions {of the form ({cf.\ Eq.\ (5) in} \cite{BP83})}
\begin{equation}\label{equa:eq5_BP83}
 p(q)=\max \{ \alpha -
\beta q^{\delta}, {0}  \},\qquad0 < \delta \le 1,
\end{equation}
{where $\alpha$ and $\beta$ are positive constants.} Note that if
$\delta =1$, {then} $p(\cdot)$ is affine; {if} $0 < \delta \le 1$,
{then} $p(\cdot)$ is convex. Assumption \ref{A:p0} implies that
$\alpha
> \chi$. Since
$$p'(q)+qp''(q) =   - \beta \delta {q^{\delta  - 1}} - \beta \delta (\delta  - 1){q^{\delta  - 1}} = -\beta \delta^2 {q^{\delta  - 1}}
\le 0 , \;\;\; 0 \le q \le \left( \frac{\alpha}{\beta}
\right)^{1/\delta}, $$ Proposition \ref{Prop:Novshek85} {implies}
that there exists a Cournot equilibrium. Through a simple
calculation we have
$$
s=\left(  \dfrac{\alpha-c}{\beta} \right)^{1/\delta},\;\;\; t=\left(
\dfrac{\alpha-c}{\beta(\delta+1)} \right)^{1/\delta}.
$$
From Corollary \ref{Coro:st_lower_bound} we know that for every
Cournot equilibrium $\mathbf{x}$,
$$
\gamma (\mathbf{x}) \ge f\left( \dfrac{-\beta\delta
t^{\delta-1}}{-\beta\delta s^{\delta-1}} \right) = f\left(
(\delta+1)^{\frac{1-\delta}{\delta}} \right ).
$$
Using the argument in Example \ref{example:eq6_BP83},  we
conclude that this lower bound also applies to the case of general
convex cost functions.} \Halmos
\end{example}

As we will see in the following example, {it is sometimes} hard to
find a closed form expression for the real number $t$. In {such
cases,} since $s$ is an upper bound for the aggregate supply at a
social optimum (cf.\ the proof of Corollary
\ref{Coro:st_lower_bound}
 in Appendix \hyperlink{B4}{B.4}), Corollary \ref{Coro:st_lower_bound}
implies that the efficiency of a Cournot candidate is at least
$f\left(\partial_+p(0)/
\partial_-p(s)\right)$. {Furthermore, in terms of} the aggregate supply at a Cournot
equilibrium $X$,
 we know that $\gamma(\bx) \ge f\left(p'(X)/
\partial_-p(s)\right)$.
}

\begin{example}\label{example:FW09}
\rm{Suppose that Assumptions \ref{A:cost}, \ref{A:optimal}, and
\ref{A:p0} hold, and {that} there is a best supplier, whose cost
function is linear with a slope $c \ge 0$. Consider inverse demand
functions  of the form ({cf.\ {p.}~8 in} \cite{FW09})
\begin{equation}\label{equa:FW09}
 p(q) = \left\{ \begin{array}{l} \displaystyle { \alpha  \left(
Q-q \right)^\beta },\qquad \qquad 0 < q \le
Q,\\[5pt]
\displaystyle 0,\qquad \qquad \qquad \qquad Q < q,
\end{array} \right.
\end{equation}
{where $Q>0$, $\alpha>0$ and $\beta \ge 1$.} Assumption \ref{A:p0}
implies that $c < \alpha Q^\beta $. Note that Corollary
\ref{Coro:mu} does not apply, because the right derivative of
$p(\cdot)$ at $Q$ is zero. Through a simple calculation we obtain
$$
s=Q-\left( \frac{c}{\alpha} \right)^{1/\beta},
$$
and
$$
p'(s)= \alpha \beta \left( \frac{c}{\alpha}
\right)^{(\beta-1)/\beta}, \qquad \partial_+p(0)= \alpha \beta
Q^{\beta-1}.
$$
  Corollary \ref{Coro:st_lower_bound} implies that for
every Cournot equilibrium $\textbf{x}$,
$$
\gamma (\textbf{x}) \ge f\left( \dfrac{\partial_+p(0)}{p'(s)}
\right) =f\left( \left(\dfrac{\alpha Q^\beta}{c}
\right)^{(\beta-1)/\beta} \right ) =f\left( \left(\dfrac{p(0)}{c}
\right)^{(\beta-1)/\beta} \right ).
$$
{Using} information {on} the aggregate demand at the equilibrium,
the efficiency bound can be refined. Since
$$
p'(X)=  \alpha \beta (Q-X)^{\beta-1},
$$
we have
\begin{equation}\label{equa:example09}
\gamma (\textbf{x}) \ge f\left( \dfrac{p'(X)}{p'(s)} \right)
=f\left( \left(\dfrac{\alpha (Q-X)^\beta}{c}
\right)^{(\beta-1)/\beta} \right ) =f\left( \left(\dfrac{p(X)}{c}
\right)^{(\beta-1)/\beta} \right ),
\end{equation}
{so that} the efficiency bound depends only on the ratio of the
equilibrium price to the marginal cost of the best supplier, and the
parameter $\beta$. For affine inverse demand functions, we have
$\beta=1$ and the bound in \eqref{equa:example09} equals $f(1)=2/3$,
which agrees {with Theorem \ref{thm:linear}.}} \Halmos
\end{example}


\old{
\begin{example} \label{ex:last}{\rm
\mg{[To Be Deleted, a tighter lower bound is given in \cite{GY05}]}
 Suppose that Assumption{s} \ref{A:cost},
\ref{A:optimal}, and \ref{A:p0} hold, and {that} there is a best
supplier, {whose} cost function is linear, with  slope $\chi > 0$.
Consider {constant elasticity inverse demand functions, of the form
(cf.\ Eq.\ (4) in} \cite{BP83})
\begin{equation}\label{equa:eq4_BP83}
 p(q)=\alpha q^{-\beta},\qquad
0 \le q,
\end{equation}
where $\alpha$ and $\beta$ are positive constants. We assume that
$\beta<1$ so that $\int_0^\epsilon p(q)\,dq < \infty$, for any
$\epsilon>0$. Through a simple calculation we have,
$$
s=\left(  \dfrac{\chi}{\alpha} \right)^{-1/\beta},\qquad  t= \left(
\dfrac{\chi}{\alpha(1-\beta)} \right)^{-1/\beta}.
$$

From Corollary \ref{Coro:st_lower_bound} we conclude that, a Cournot
equilibrium $\mathbf{x}$, if {one} exists, must satisfy
$$
\gamma (\textbf{x}) \ge f\left( \dfrac{-\alpha \beta
t^{-\beta-1}}{-\alpha \beta s^{-\beta-1}} \right) = f\left(
(1-\beta)^{\frac{-\beta-1}{\beta}} \right ).
$$
By the same argument in Example \ref{example:eq6_BP83},  this lower
bound also applies to the case of general convex cost functions.}
\Halmos
\end{example}
}

\section{Monopoly and Social Welfare}\label{sec:mono}

In this section we study the special case where $N=1$, so that we
are dealing with a single, monopolistic, supplier. As we explain,
this case also covers a setting where multiple suppliers collude to
maximize their total profit. By using the additional assumption that
$N=1$, we obtain a sharper (i.e., larger) lower bound, in Theorem
\ref{thm:mono}. We then establish lower bounds on the efficiency of
monopoly outputs that do not require knowledge of $X$ and $X^S$.

In a Cournot oligopoly, the maximum possible profit earned by all
suppliers (if they collude) is an optimal solution to the following
optimization problem,
\begin{equation}\label{equa:max_profit}
\begin{array}{l}
 \displaystyle{\rm{maximize}}\;\;p\left(\sum\limits_{n = 1}^N x_n\right)  \cdot \sum\limits_{n = 1}^N x_n    - \sum\limits_{n = 1}^N {C_n({x_n})}  \\
 \displaystyle{\rm{subject}}\;{\rm{to}}\;\;{x_n} \ge 0,\;\;\;n = 1, \ldots ,N. \\
 \end{array}
 \end{equation}
We use $\bx^P=(x^P_1,\ldots,x^P_N)$ to denote an optimal solution to
\eqref{equa:max_profit} ({a \it monopoly output}), and let $X^P=
\sum\limits_{n = 1}^N {x_n^P } $.

It is not hard to see that the aggregate supply at a monopoly
output, $X^P$, is also a Cournot equilibrium in a modified
model with a single supplier ($N$=1) and a cost function 
given by
\begin{equation}\label{equa:mono2}
\begin{array}{l}
\displaystyle \;\;\;\;\;\;\;\;\;\;\;\;\;\;\;\;\;\;\;\;\;\;\; \overline C(X)=\inf \,\sum\limits_{n = 1}^N {C_n({x_n})},
\\[5pt]
\displaystyle {\rm{subject}}\;{\rm{to}}\;\;{x_n} \ge 0,\;\;\;n = 1, \ldots ,N;\;\;\; \sum\limits_{n = 1}^N x_n =X. \\
 \end{array}
 \end{equation}
Note that $\overline C(\cdot)$ is convex (linear) when the
$C_n(\cdot)$ are convex (respectively, linear). Furthermore, the
social welfare
 at the monopoly output $\bx^P$, is the same as that achieved at the
 Cournot equilibrium, $x_1=X^P$, in the modified model.
Also, the socially optimal value of $X$, as well as the resulting
social welfare is the same for the $N$-supplier model and the above
defined modified model with $N=1$. Therefore, the efficiency of the
monopoly output equals the efficiency of the Cournot equilibrium of
the modified model. To lower bound the efficiency of monopoly
outputs resulting from multiple colluding suppliers, we can (and
will) restrict to the case with $N=1$.

\begin{theorem}\label{thm:mono}
Suppose that Assumptions \ref{A:cost}-\ref{A:p0} hold, and the
inverse demand function is convex. Let $\bx^S$ and $\bx^P$  be a
social optimum and a monopoly output, respectively. Then, the
following hold.
\begin{enumerate}

\item[(a)] If $p(X^P)=p(X^S)$, then $\gamma(\bx^P)=1 $.

\item[(b)] If $p(X^P)\ne p(X^S)$, let
$c=|p'(X^P)|$, $d=|(p(X^S)-p(X^P))/(X^S-X^P)|$, and $\overline c=
c/d$. We have $\overline c \ge 1$ and
\begin{equation}\label{equa:thm_mono}
\gamma(\bx^P) \ge \dfrac{3}{3+ \overline c} .
\end{equation}

\item[(c)] The bound is tight at $\overline c=1$, i.e., there exists a model
with $\overline c=1$ and a monopoly output whose efficiency is
$3/4$.
\end{enumerate}
\end{theorem}

 \begin{figure}\label{Fig:mono}
    \includegraphics[width=10.5cm]{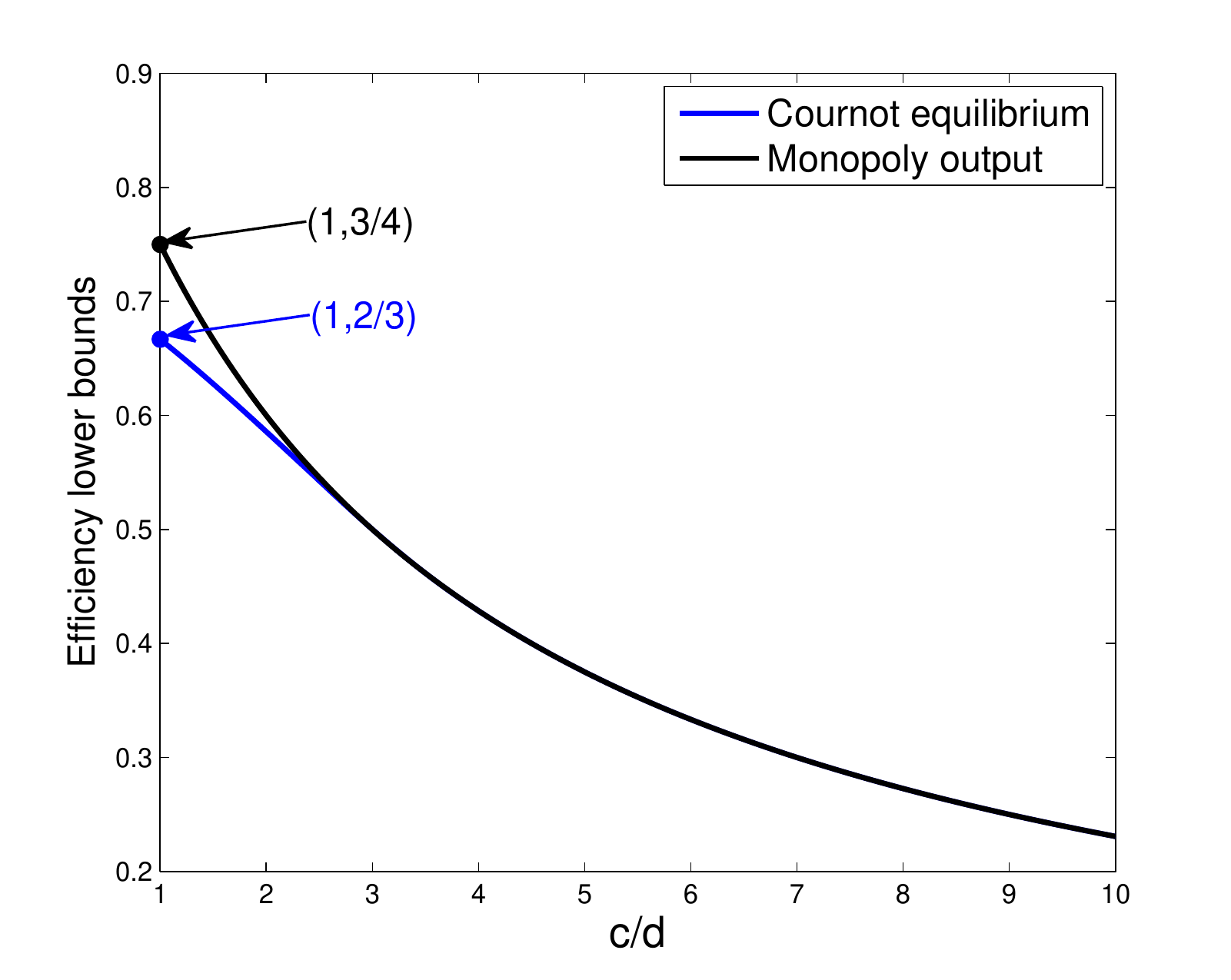}
    \centering
    \caption{Comparison of the lower bounds on the efficiency of Cournot equilibria and monopoly outputs for the case of convex inverse demand functions.}
  \end{figure}

The proof for Theorem \ref{thm:mono} can be found in Appendix
\hyperlink{C}{C.1}. Fig.\ 4 compares the efficiency lower bounds
established for Cournot equilibria with that for monopoly outputs.
For $\overline c=1$, both efficiency bounds are tight and it is
possible for a monopoly output to achieve a higher efficiency than
that of a Cournot equilibrium, as shown in the following example.

\begin{example}\label{example:mono}  \rm{
Consider the model introduced in the proof {of} part (c) of Theorem
\ref{thm:linear}. The inverse demand function is
$p(q)=\max\{1-q,0\}$. The cost functions are linear:
$$
C^N_1(x_1)=0, \qquad C^N_n(x_n)= \left( 1/3- \dfrac{1/3}{N-1}
\right) x_n, \qquad n=2,\ldots,N.
$$
If all suppliers collude to maximize the total profit, the output
will be
$$
x^P_1= 1/2, \qquad x^P_n= 0, \quad n=2,\ldots,N,
$$
and $\gamma(\bx^P) = 3/4$. On the other hand, it can be verified
that the vector
$$
x_1= 1/3, \qquad x_n= \dfrac{1/3}{N-1}, \quad n=2,\ldots,N,
$$
 is a
Cournot equilibrium. For any $N \ge 2$, a simple calculation shows
that the associated efficiency is $(6N-4)/(9N-9)$. For example, in a
model with $N=10$, the efficiency of the Cournot equilibrium is less
than that of the monopoly output, i.e.,
 $\gamma(\bx^P) = 3/4 > 56/81=\gamma(\bx)$.   \Halmos
}
\end{example}

 The above example agrees with the observation
in earlier works that a monopoly output is not necessarily less
efficient than an equilibrium resulting from imperfect competition
\citep{CL69,CR71}.

 We now derive a result similar to Corollary
\ref{Coro:st_lower_bound}, and then apply it to a numerical example
with the same inverse demand function as in Example
\ref{example:eq6_BP83}.

\begin{corollary}\label{Coro:st_lower_bound_mono}
Suppose that Assumptions \ref{A:cost}-\ref{A:p0} hold, and that
$p(\cdot)$ is convex. Let
 $s$ and $t$ be the real numbers defined in
 \eqref{equa:t}. If
$\partial_-p(s)<0$, then for any monopoly output $\bx^P$, we have
$$
\gamma(\bx^P) \ge
\dfrac{\partial_-p(s)}{3\partial_-p(s)+\partial_+p(t)} .
$$
\end{corollary}

Corollary \ref{Coro:st_lower_bound_mono} is proved in Appendix
\hyperlink{C2}{C.2}.

\begin{example}\label{example:eq6_BP83_mono}
\rm{Suppose that Assumptions \ref{A:cost}, \ref{A:optimal}, and
\ref{A:p0} hold, and that there is a best supplier, whose cost
function  is linear with a slope $\chi \ge 0$. Consider inverse
demand functions of the form in \eqref{equa:eq6_BP83}. \old{
   (eq.6,
\cite{BP83}),
\begin{equation}\label{equa:eq6_BP832}
p(q)=\alpha-\beta \log q,\;\;\;\alpha, \; \beta >0, \;\;\; 0<q
<\exp(\alpha/\beta).
\end{equation}}
 Through a simple calculation we have
$$
s=\exp\left(\dfrac{\alpha-\chi}{\beta}\right),\qquad t=\exp\left(
{\dfrac{\alpha-\beta-\chi}{\beta} }  \right),
$$
and
$$
 \dfrac{p'(t)}{
p'(s)} = \dfrac{\exp\left((\alpha-\chi)/\beta\right)}{\exp\left(
(\alpha-\beta-\chi)/\beta \right)}=\exp( 1 ).
$$
According to Corollary \ref{Coro:st_lower_bound_mono}, for every
monopoly output $\bx^P$ we have,
\begin{equation}\label{equa:example81}
\gamma (\bx^P) \ge  3/(3+\exp( 1 )) {=0.525}.
\end{equation}

Using the argument in Example \ref{example:eq6_BP83}, we conclude
that this efficiency bound also
  applies to the case of nonlinear (convex) cost functions.} \Halmos
\end{example}

\section{Conclusion}\label{sec:con}
It is well known that Cournot oligopoly can yield arbitrarily high
efficiency loss in general; for details, see \cite{J04}. For Cournot
oligopoly with convex market demand and cost functions, results such
as those provided in Theorem \ref{thm:convex} show that the
efficiency loss of a Cournot equilibrium can be bounded away from
zero by a function of a scalar parameter that captures quantitative
properties of the inverse demand function. With additional
information on the cost functions, the efficiency lower bounds can
be further refined. Our results apply to various convex inverse
demand functions that have been considered in the economics
 literature.


\bibliographystyle{nonumber}

\ACKNOWLEDGMENT{This research was supported in part by the National
Science Foundation under grant CMMI-0856063 and by a Graduate
Fellowship from Shell.}

\newpage

\begin{APPENDICES}

\pdfbookmark[0]{Proofs of the results in Section
\ref{sec:property}}{A}

\section*{Appendix A: Proofs of the results in Section
\ref{sec:property}}

\subsection*{A.1. Proof of Proposition \ref{Prop:sameprice}}
\pdfbookmark[1]{Proof of Proposition \ref{Prop:sameprice}}{A1}

Suppose not, in which case there exist two optimal solutions,
$\bx^S$ and $\overline \bx^S$, such that $p(X^S) \ne p(\overline
X^S)$. Without loss of generality, we assume that $p(X^S) >
p(\overline X^S)$. Since $p(\cdot)$ is nonincreasing, we {must} have
\old{[that]} $X^S<\overline X^S$. For all $n$ such that $\overline
x_n^S>0$, the optimality conditions (\ref{equa:optimality}) {yield}
$$
C'_n(\overline x_n^S)=p(\overline X^S)<p(X^S) \le C'_n(x_n^S).
$$
{Using the convexity of the cost functions, we obtain}
$$
\overline x^S_n < x_n^S,\qquad {\rm{if}}\; \overline x^S_n>0,
$$
This contradicts the assumption that $X^S<\overline X^S$, {and} the
desired result follows.

\pdfbookmark[1]{Proof of Proposition \ref{Prop:derivative}}{A2}

\subsection*{A.2. Proof of Proposition
\ref{Prop:derivative}}\label{sec:A2}

Let $\textbf{x}$ be a Cournot candidate with $X>0$. The conditions
(\ref{equa:nece1})-\eqref{eq:necsec} applied to some $n$ with
$x_n>0$, imply that
\[
p\left( X \right) + {x_n} \cdot \partial_- p \left( X \right) \ge
p\left( X \right) + {x_n} \cdot \partial_+ p\left( X \right).
\]
On the other hand, since $p(\cdot)$ is  convex, we  have $
\partial_- p\left( X \right) \le \partial_+ p\left( X \right)$.
Hence, $\partial_- p\left( X \right)  = \partial_+ p\left( X
\right)$, as claimed.

\pdfbookmark[1]{Proof of Proposition \ref{Prop:positive}}{A3}

\subsection*{A.3. Proof of Proposition \ref{Prop:positive}}\label{sec:A3}

{Using Assumption \ref{A:p0}, we can choose some $k$ such that
$p(0)>C'_k(0)$. The right derivative with respect to $x_k$ of the
obejctive function in (\ref{equa:optimal}), evaluated at
$\bx=(0,0,\ldots,0)$, is $p(0)-C_k(0)>0$. Hence the optimal value of
the objective is strictly larger than the zero value obtained when
$\bx= (0,0,\ldots,0)$. Thus, the optimal social welfare is
positive.}

Now consider the social welfare achieved at a Cournot {candidate}
$\bx=(x_1,\ldots,x_N)$.  Because of Assumption \ref{A:p0},
Proposition \ref{Prop:trivial} applies, and we have $X>0$. For any
supplier $n$ such that $x_n>0$, {the necessary conditions}
{(\ref{equa:nece1})}  {and the fact that $p(\cdot)$ is nonincreasing
imply that}  ${C'_n}({x_n}) \le p\left( X \right)$. {Suppose that
${C'_n}({x_n}) = p\left( X \right)$ for every supplier $n$  with
$x_n>0$.}  Then,  {the necessary conditions
(\ref{equa:nece1})-\eqref{eq:necsec} imply that $\bx$ satisfies} the
{sufficient} optimality conditions in (\ref{equa:optimality}).
{Thus,} $\textbf{x}$ is socially optimal {and} the desired result
follows.

 {Suppose now that} there exists some supplier $n$  {with} $x_n>0$ and
${C'_n}({x_n}) < p\left( X \right)$.  {Then,}
\begin{equation}
\sum_{n = 1}^N {C'_n({x_n})x_n}  < Xp\left( X \right)  \le \int_0^{X
} {p(q)\, dq},
 \end{equation}
where the last inequality holds because the function $p(\cdot)$ is
nonincreasing. Since for each $n$, $C_n(\cdot)$ is convex and
nondecreasing, {with $C_n(0)=0$,} we have
\begin{equation}\label{equa:positive}
\sum_{n = 1}^N {C_n({x_n})}  \le \sum_{n = 1}^N {C'_n({x_n})x_n}  <
Xp\left( X \right)  \le \int_0^{X } {p(q)\,dq}.
 \end{equation}
Hence, the social welfare at the Cournot {candidate}, \(\int_0^{X }
{p(q)\, dq}  - \sum_{n = 1}^N {C_n({x_n})} \), is positive.

\pdfbookmark[1]{Proof of Proposition \ref{Prop:linear}}{A4}

\subsection*{A.4. Proof of Proposition
\ref{Prop:linear}}\label{sec:A4}

{We first observe that the vector $\mathbf{x}$ satisfies the
necessary conditions \eqref{equa:nece1}-\eqref{eq:necsec} {for} the
modified {model}. Hence,  the vector $\mathbf{x}$ is a Cournot
candidate for the modified {model.} It is also not hard to see that
Assumptions \ref{A:cost} and \ref{A:demand} are satisfied {by} the
modified {model}. Since $\alpha_n \ge C'_n(0)$ for every $n$,
Assumption \ref{A:optimal} {also} holds in the modified {model}.}

{We now show that Assumption \ref{A:p0} holds in the modified
{model}, i.e., {that} $p(0)>\min_n\{\alpha_n\} $.  Since the vector
$\mathbf{x}$ is a Cournot candidate {in} the original {model,}
Proposition \ref{Prop:trivial} implies that $X>0$, {so that} there
exists some $n$ {for which} $x_n>0$. From the necessary condition
\eqref{equa:nece1} we have that $\alpha_n \le p(X)$. Furthermore,
{if} $\alpha_n = p(X)$, then \(\partial_{-}p(X)=0 \), {and} the
convexity of $p(\cdot)$ implies that $\partial_{+}p(X)=0 $. Hence,
the vector $\mathbf{x}$ satisfies the optimality condition
\eqref{equa:optimality}, and thus is socially optimal in the
original {model.} Under {our} assumption that $\mathbf{x}$ is not
socially optimal in the original {model}, we conclude that $\alpha_n
< p(X)$, which implies that Assumption \ref{A:p0} holds in the
modified {model}. }

Let $\bx^S$ be an optimal solution to (\ref{equa:optimal}) in the
original model. Since $\bx^S$ satisfies the optimality conditions in
\eqref{equa:optimality} for the modified model, it remains a social
optimum in the modified model. {In the modified model, since
Assumptions \ref{A:cost}-\ref{A:p0} hold, the efficiency of the
vector $\mathbf{x}$ is well defined and given by}
\begin{equation}\label{eq:lbl}
\overline \gamma(\mathbf{x})=\dfrac{{\int_0^{{X} } {p(q)\,dq} -
\sum\nolimits_{n = 1}^N {{\alpha _n}x_n} }}{
 { \int_0^{{X^S} }
{p(q)\,dq}  - \sum\nolimits_{n = 1}^N {{\alpha _n}x^S_n} }}.
\end{equation}
{Note that the denominator on the right-hand side of \eqref{eq:lbl}
is the optimal social welfare and the numerator  is the social
welfare achieved at the Cournot candidate $\mathbf{x}$,  in the
modified model. Proposition \ref{Prop:positive} implies that both
the denominator and the numerator  on the right-hand side of
\eqref{eq:lbl} are positive. {In particular, $\overline
\gamma(\mathbf{x})>0$.}}

 Since $C_n(\cdot)$ is
convex, we have
$$
C_n(x^S_n)  - C_n(x_n)- \alpha_n (x^S_n -x_n) \geq 0,\qquad  n
=1,\ldots,N.
$$
Adding a nonnegative quantity to the denominator cannot increase the
ratio and, therefore,
\begin{equation}\label{equa:linear1}
{1 \geq\ } \gamma(\textbf{x}) =\dfrac{{ \int_0^{{X} } {p(q)\,dq} -
\sum\nolimits_{n = 1}^N {C_n({x_n})} }}{{\int_0^{{X^S} } {p(q)\,dq}
- \sum\nolimits_{n = 1}^N {C_n(x_n^S)} }} \ge \dfrac{{\int_0^{{X} }
{p(q)\,dq}  - \sum\nolimits_{n = 1}^N {C_n({x_n})} }}{ \int_0^{{X^S}
} {p(q)\,dq}-\sum\nolimits_{n = 1}^N {\left( {{\alpha _n}(x_n^S -
{x_n}) + {C_n}({x_n})} \right)} } {>0}.
\end{equation}

 Since $C_n(\cdot)$ is convex and nondecreasing, with
$C_n(0)=0$, we also have
\begin{equation}\label{equa:linear3}
\sum_{n = 1}^N {{C_n}({x_n})}  - \sum_{n = 1}^N {{\alpha _n}{x_n}}
\le 0.
\end{equation}
{Since the right-hand side of (\ref{equa:linear1}) {is in the
interval $(0,1]$,} adding the left-hand side of  Eq.\
\eqref{equa:linear3} (a nonpositive quantity) to both the numerator
and the denominator {cannot increase} the ratio, as long as the
numerator remains nonnegative. The numerator remains indeed
nonnegative because it becomes the same as the numerator in the
expression \eqref{eq:lbl} for $\overline{\gamma}(\bx)$. We obtain
\begin{eqnarray}
{\gamma(\bx)} &{\geq}&\dfrac{{\int_0^{{X} } {p(q)\,dq}  -
\sum\nolimits_{n = 1}^N {C_n({x_n})} }}{\int_0^{{X^S} }
{p(q)\,dq}-\sum\nolimits_{n = 1}^N {\left( {{\alpha
_n}(x_n^S - {x_n}) + {C_n}({x_n})} \right)}} \nonumber \\
&  \ge& \dfrac{{\int_0^{{X} } {p(q)\,dq}  - \sum\nolimits_{n = 1}^N
{C_n({x_n})} + \left( {\sum\nolimits_{n = 1}^N {{C_n}({x_n})}  -
\sum\nolimits_{n = 1}^N {{\alpha _n}{x_n}} } \right)
}}{\int_0^{{X^S}} {p(q)dq}-\sum\nolimits_{n = 1}^N {\left( {{\alpha
_n}(x_n^S - {x_n}) + {C_n}({x_n})} \right)}+ \left( {\sum\nolimits_{n = 1}^N {{C_n}({x_n})}  - \sum\nolimits_{n = 1}^N {{\alpha _n}{x_n}} } \right)} \nonumber \\
&=&\dfrac{{\int_0^{{X} } {p(q)\,dq}  - \sum\nolimits_{n = 1}^N
{{\alpha _n}x_n} }} {{\int_0^{{X^S} } {p(q)\,dq}  - \sum\nolimits_{n
= 1}^N {{\alpha _n}x^S_n} }}
\nonumber \\
\old{&\ge& \dfrac{{\int_0^{{X} } {p(q)\,dq}  - \sum\nolimits_{n =
1}^N {{\alpha _n}x_n} }}{{{{{\sup} }_{Y \ge 0}}\left\{ {\int_0^Y
{p(q\,)dq - \left( {{{\min }_n}{\alpha _n}} \right)Y} } \right\}}}
\label{eq:rat2}\\}
 &{=}& {\overline{\gamma}(\bx)}. \nonumber
\label{equa:linear4}
 \end{eqnarray}
The desired result follows. \old{from {(\ref{eq:lbl})},
(\ref{equa:linear1}) and (\ref{equa:linear4}).}}

\pdfbookmark[1]{Proof of Proposition \ref{Prop:less}}{A5}

\subsection*{A.5. Proof of Proposition
\ref{Prop:less}}\label{sec:A5}

By Assumption \ref{A:p0},  $p(0) > \min_n\{C'_n(0)\}$. According to
Proposition \ref{Prop:trivial}, we have $X>0$. Since $p(\cdot)$ is
nonincreasing, {the} conditions in  (\ref{equa:nece1}) imply that
$$
C'_n(x_n^{}) \le p\left( X \right),\qquad {\rm{if}}\; x_n>0.
$$
Suppose that $X \geq  X^S$. Since the inverse demand function is
nonincreasing and $p(X)\ne p(X^S)$, we have $p (X) < p(X^S)$ {and
$X>X^S$.}
 For
every supplier $n$ {with} $x_n>0$, we have
$$
C'_n(x_n^{}) \le p\left( X \right)  < p(X^S) \le C'_n(x_n^{S}),
$$
where the last inequality follows from (\ref{equa:optimality}).
Since, $C_n(\cdot)$ is  convex, the above inequality implies that
$$
x_n < x_n^S,\qquad {\rm{if}}\; x_n>0,
$$
{from which we obtain $X<X^S$. Since we had assumed that $X\geq
X^S$, we have  a contradiction.}

{The preceding argument establishes that $X<X^S$. Since $p(\cdot)$
is noincreasing and $p(X)\neq p(X^S)$, we must have $p(X)>p(X^S)$.}

\pdfbookmark[1]{Proof of Proposition \ref{Prop:equal}}{A6}

\subsection*{A.6. Proof of Proposition \ref{Prop:equal}}\label{sec:A6}

Since Assumption \ref{A:p0} holds, Proposition \ref{Prop:trivial}
{implies} that $X>0$. Since $p(\cdot)$ is convex, {Proposition
\ref{Prop:derivative} shows that $p(\cdot)$ is differentiable at $X$
and the necessary conditions in (\ref{equa:nece}) are satisfied.}

{We will now} prove that $p'(X)=0$. Suppose not, {in which case} we
have $p'(X)<0$.
 For every $n$ such that $x_n>0$, from the convexity of $C_n(\cdot)$
 and the conditions in
(\ref{equa:nece}), we have
$$
C'_n(0) \le C'_n(x_n) < p(X) =p(X^S).
$$
{Then, the social optimality conditions \eqref{equa:optimality}
imply} that $x_n^S>0$. {It follows that}
$$
C'_n(x_n) <p(X^S)=C'_n(x^S_n),
$$
where the last equality follows from the optimality conditions in
(\ref{equa:optimality}). Since $C_n(\cdot)$ is convex,
 {we conclude} that $ x_n < x^S_n. $ Since
{this is true} for every $n$ such that $x_n>0$, we {obtain} $ X <
X^S. $ Since the function $p(\cdot)$ is nonincreasing, {and we have}
$p'(X)<0$ and $X<X^S$, {we obtain $p(X)>p(X^S)$, which contradicts
the assumption that} $p(X)=p(X^S)$. {The contradiction shows} that
$p'(X)=0$.

Since $p'(X)=0$ and the Cournot {candidate} $\bx$ satisfies the
necessary conditions in (\ref{equa:nece}), it {also satisfies} the
optimality conditions in (\ref{equa:optimality}). Hence, $\bx$ is
socially optimal and the desired result follows.

\pdfbookmark[0]{Proofs of the results in Sections
\ref{sec:affine}-\ref{sec:app}}{B}

\section*{Appendix B: Proofs of the results in Sections \ref{sec:affine}-\ref{sec:app}}

\pdfbookmark[1]{Proof of Theorem \ref{thm:linear}}{}

\hypertarget{B1}{}

\subsection*{B.1. Proof of Theorem \ref{thm:linear}}\label{sec:B1}

{We note that part (d) is an immediate consequence of the expression
for $g(\beta)$, and we concentrate on the remaining parts.} {Since
the inverse demand function is convex, Proposition
\ref{Prop:derivative} shows that any Cournot equilibrium satisfies
the necessary conditions (\ref{equa:nece}). If $X>b/a$, then \old{we
have that} $p(X)=p'(X)=0$. {In that case,} the necessary conditions
(\ref{equa:nece}) imply the optimality conditions
(\ref{equa:optimality}). We conclude that $\bx$ is socially optimal.
}

We now assume that $X\le b/a$.  Proposition \ref{Prop:derivative}
shows that $p'(X)$ exists, and thus $X< b/a$. Since $p'(X) =-a <0$,
Proposition \ref{Prop:equal} {implies} that $p(X) \ne p(X^S)$, for
any social optimum $\bx^S$. Hence, $\bx$ is not socially optimal.

{As discussed in Section \ref{sec:linear-worst}, to derive a lower
bound, it suffices to consider the case of linear cost functions,
and obtain a lower bound on the worst case efficiency of Cournot
candidates, that is, vectors that satisfy
(\ref{equa:nece1})-(\ref{eq:necsec}).} We {will} therefore assume
that $C_n(x_n)=\alpha_n x_n$ for {every} $n$. Without loss of
generality, we {also} assume that $\alpha_1 = \min_n \{\alpha_n\}$.
{We consider separately the two cases where
$\alpha_1=0$ or $\alpha_1>0$, respectively.\\

\noindent  \textbf{The case where $\alpha_1=0$}

{In this case, the socially optimal supply is $X^S=b/a$ and the
optimal social welfare is}
\[
\int_0^{{b}/{a}} {p(q)\,dq}  - 0 = \int_0^{{b}/{a}} {( - ax + b)\,
dx = \frac{{{b^2}}}{{2a}}}.
 \]
{Note also that $\beta=aX/b$.}

{Let $\bx$ be a Cournot candidate. Suppose first that $x_1=0$. In
that case, the necessary conditions $0=\alpha_1\geq p(X)$ imply that
$p(X)=0$. For $n\neq 1$, if $x_n>0$, the necessary conditions yield
$0\leq \alpha_n=p(X)-x_n a=-x_n a$, which implies that $x_n=0$ {for}
all $n$. But then, $X=0$, which contradicts the fact $p(X)=0$. We
conclude that $x_1>0$.}

{Since {$x_1>0$,} the necessary conditions {\eqref{equa:nece}}}
yield ${0}=\alpha_1=b-aX- ax_1$, so that
\begin{equation}\label{equa:thm_linear2}
 x_1 = -X+\dfrac{b}{a}.
 \end{equation}
In particular, $X< b/a =X^S$, and $\beta < 1$. Furthermore,
$$0\leq \sum_{n=2}^N x_n=X-x_1 =2X -\dfrac{b}{a},$$
from which we conclude that $\beta=aX/b \geq 1/2$. }

{Note that for $n=1$ we have $\alpha_n x_n=0$. For $n\neq 1$,
whenever $x_n>0$, we have $\alpha_n = p(X)-ax_n$, {so that}
$\alpha_n x_n = (p(X)-ax_n)x_n$. {The} social welfare associated
with $\bx$ is
\begin{eqnarray}
\int_0^X {p(q)\,dq - \sum_{n = 1}^N {{\alpha _n}{x_n}} } &=&
bX-\frac{1}{2}aX^2- \sum_{n=2}^N (p(X)-ax_n) x_n\nonumber\\
&\geq& bX-\frac{1}{2}aX^2- p(X)\sum_{n=2}^N  x_n \label{eq:loose}\\
&=&bX-\frac{1}{2}aX^2-  (b-aX)(X-x_1)\nonumber\\
&=&bX-\frac{1}{2}aX^2-  (b-aX)\Big(2X-\frac{b}{a}\Big)\nonumber\\
&=& \dfrac{3}{2}a{X^2} + \dfrac{{{b^2}}}{a} - 2bX.\nonumber
\end{eqnarray}
We divide by $b^2/2a$ (the {optimal} social welfare) and obtain
$$\gamma(X) \geq
\dfrac{{2a}}{{{b^2}}}\left( {\dfrac{3}{2}a{X^2} + \dfrac{{{b^2}}}{a}
- 2bX} \right) =3 \beta^2 - 4 \beta +2.$$ This proves the claim in
part (b) of the theorem.}\\

\noindent  \textbf{Tightness}

{We observe that the lower bound on the social welfare associated
with $\bx$ made use, {in Eq.\ \eqref{eq:loose},} of the inequality
$\sum_{n=2}^Nx_n^2\geq 0$. This inequality becomes an equality,
asymptotically, if we let $N\to\infty$ and $x_n{=O(1/N)}$ for $n\neq
1$. This motivates the proof of tightness (part (c) of the theorem)
given below.}

{We are given some $\beta\in [1/2,1)$ and construct an $N$-supplier
model ($N \ge 2$) with $a=b=1$, and the following linear cost
functions:}
$$
C^N_1(x_1)=0, \qquad C^N_n(x_n)= \left( p(X)- \dfrac{2X-{1}}{N-1}
\right) x_n, \qquad n=2,\ldots,N.
$$
It can be verified that {the variables}
$$
x_1= -X +b/a, \qquad x_n= \dfrac{2X-b/a}{N-1}, \quad n=2,\ldots,N,
$$
form a Cournot equilibrium. A simple calculation
(consistent with the intuition given earlier) shows that as $N$
increases to infinity, the sum $\sum_{n=2}^Nx_n^2$ goes to zero and
the associated efficiency converges to $g(\beta)$.\\

\noindent  \textbf{The case where  $\alpha_1>0$}

We now consider the case where $\alpha_n>0$ for every $n$. {By
rescaling the cost coefficients and permuting the supplier indices,
we can} assume that $\min_n \{\alpha_n\}=\alpha_1=1$.  By Assumption
\ref{A:p0}, we have $b>1$.

At the social optimum, we must have $p(X^S)=\alpha_1=1$ and thus
$X^S=(b-1)/a$. The optimal social welfare is
\[
\dfrac{(b-1)(b+1)}{2a}-\dfrac{b-1}{a}=\dfrac{(b-1)^2}{2a}.
\]
{Note also that $\beta=aX/(b-1)$.}

{Similar to the proof for the case where $\alpha_1=0$, we can show
that $x_1>0$ and therefore $1 = p (X) -
 a{x_1}=b-aX-ax_1$, so that}
\[
 x_1 = -X+\dfrac{b-1}{a} >0,
 \]
which implies that $\beta<1$.
In particular,
\[
X < \dfrac{b-1}{a}{=X^S.}
 \]
{Furthermore,
$$0\leq \sum_{n=2}^N x_n =X-x_1= 2X -\dfrac{b-1}{a},$$
from which we conclude that $\beta=aX/(b-1)\geq 1/2$.}

{A calculation similar to the one for the case where $\alpha_1=0$
yields
\begin{eqnarray*}
\int_0^X {p(q)\,dq - \sum_{n = 1}^N {{\alpha _n}{x_n}} } &=&
bX-\frac{1}{2}aX^2{- x_1} -\sum_{n=2}^N (p(X)-ax_n) x_n\\
&\geq& bX-\frac{1}{2}aX^2{+X-\frac{b-1}{a}} - p(X)\sum_{n=2}^N  x_n\\
&=&bX-\frac{1}{2}aX^2{+X-\frac{b-1}{a}}  -(b-aX)(X-x_1)\\
&=&bX-\frac{1}{2}aX^2 {+X-\frac{b-1}{a}}  -(b-aX)\Big(2X-\frac{b-1}{a}\Big)\\
 &=& \dfrac{3}{2}a{X^2} + \dfrac{{{(b-1)^2}}}{a} - 2(b-1)X.
\end{eqnarray*}
 After dividing with the value of
the social welfare, we obtain $g(\beta)$, as desired.}

\pdfbookmark[1]{Proof of Theorem \ref{thm:convex}}{B2}

\hypertarget{B2}{}

\subsection*{B.2. Proof of Theorem
\ref{thm:convex}}\label{sec:B2}

 {Let $\bx$ be a Cournot candidate.} According to Proposition
\ref{Prop:equal},  if $p(X)=p(X^S)$, then the efficiency of the
Cournot {candidate} must equal one, {which proves part (a). To prove
part (b), we assume that} $p(X) \ne
 p(X^S)$. {By} Proposition \ref{Prop:sameprice}, \old{we conclude that} the Cournot candidate $\mathbf{x}$ cannot be socially optimal, {and, therefore, $\gamma(\bx)<1$.}

{We have shown in Proposition \ref{Prop:linear} that if all cost
functions are replaced by linear ones, the vector $\mathbf{x}$
remains a Cournot candidate, and Assumptions \ref{A:cost}-\ref{A:p0}
still hold.  Further, the efficiency of $\bx$ {cannot increase}
after all cost functions are replaced by linear ones. Thus,} {to
lower bound} the worst case efficiency loss, {it suffices to derive
a lower bound for the efficiency of Cournot candidates for the case
of} linear cost functions. {We therefore assume that}
$C_n(x_n)=\alpha_n x_n$ for each $n$. Without loss of generality, we
{further} assume that $\alpha_1 = \min_n \{\alpha_n\}$. {Note that,
by Assumption \ref{A:p0}, we have $p(0)>\alpha_1$.} We will prove
the theorem by considering {separately} the cases where $\alpha_1=0$
and $\alpha_1>0$.

{We will rely on} Proposition \ref{Prop:piecewise_linear},
{according to which} the efficiency of a Cournot candidate $\bx$ is
lower bounded by the efficiency $\gamma^0(\bx)$ {of $\bx$ in a model
involving the piecewise linear and convex inverse demand function
function of the form in the definition of $p^0(\cdot)$. Note that}
since $p(X) \ne
 p(X^S)$, we have that $d>0$. For conciseness, we let $y=p(X)$ throughout the
 proof.\\

\old{The main idea of the proof, as in \cite{J04}, is to first use
the results of Section \ref{sec:linear-worst} to restrict to linear
cost functions; then use the result of Proposition
\ref{Prop:piecewise_linear} to restrict to a finitely parametrized
family of inverse demand functions; and finally search (and
optimize) for a worst-case model, over a finite-dimensional family
of cost and demand functions, which can be done in closed form.}

\noindent \textbf{The case $\alpha_1=0$}

Let $\mathbf{x}$ be a Cournot candidate {in} the original {model}
with linear cost functions  and the inverse demand function
$p(\cdot)$. {By} Proposition \ref{Prop:derivative}, \old{we conclude
that} $\mathbf{x}$ satisfies the necessary conditions
(\ref{equa:nece}), with respect to the original inverse demand
function $p(\cdot)$.  Suppose first that $x_1 = 0$. The second
inequality in (\ref{equa:nece}) implies that $p(X)=0$. {On the other
hand, Assumption \ref{A:p0} and} Proposition \ref{Prop:trivial}
{imply} that $X>0$. {Thus,} there exists some $n$ such that $x_n>0$.
The first equality in (\ref{equa:nece}) yields,
$$
0\le \alpha_n=p(X) + x_n p'(X)  {=x_n p'(X)} \le 0,
$$
which implies {that} $ p'(X)=0$. Then, the vector $\mathbf{x}$
satisfies the optimality conditions in (\ref{equa:optimality}), and
{is thus} socially optimal in the original {model}. This contradicts
\old{with} the fact that $p(X)\ne p(X^S)$ {and shows that we must
have} $x_1 > 0$.

If $p'(X)$ {were equal to zero,} then the necessary conditions
(\ref{equa:nece}) {would} imply the optimality conditions
(\ref{equa:optimality}), and $\mathbf{x}$ {would be} socially
optimal in the original {model}. Hence, we {must} have $p'(X)<0$
{and} $c>0$. The first equality in (\ref{equa:nece}) yields $y >0$,
\(
 x_1 = y/c \), {and $X\geq y/c$. We also}  have
\begin{equation}\label{equa:thm51}
0 \le \sum_{n=2}^N x_n =X-\frac{y}{c}.
\end{equation}

{From Proposition \ref{Prop:piecewise_linear}, the efficiency
{$\gamma^0(\bx)$ of $\bx$ in the modified model}
 cannot be more than {its efficiency $\gamma(\bx)$} in
the original {model}. Hence, to prove the second part of the
theorem, it suffices to show that $\gamma^0(\mathbf{x}) \ge
f(\overline c)$, for any Cournot candidate with $c/d=\overline c$.}

{The optimal social welfare in the modified {model} is
\begin{equation}\label{equa:optimal_welfare0}
\int_0^{\infty} {p^0(q)dq}  - 0 = \int_0^{X+y/d} {p^0(q)dq}  - 0 =
\dfrac{{{y^2}}}{{2d}}+ \dfrac{(2y+cX)X}{2}.
 \end{equation}}

{Note that for $n = 1$ we have $\alpha_n x_n = 0$. For $n \ge 2$,
whenever $x_n > 0$, from the first equality in (\ref{equa:nece}) we
have $\alpha_n=y-x_n c$ and $\alpha_n x_n=(y-x_n c)x_n$. Hence, in
the modified {model}, the social welfare associated with
$\mathbf{x}$ is
\[
\begin{array}{l}
\displaystyle
\int_0^X {{p^0}(q)dq - \sum\nolimits_{n = 1}^N {{\alpha _n}{x_n}} }  = \dfrac{{(2y + cX)X}}{2} - \sum\nolimits_{n = 2}^N {(y - {x_n}c){x_n}} \\
\displaystyle\;\;\;\;\;\;\;\;\;\;\;\;\;\;\;\;\;\;\;\;\;\;\;\;\;\;\;\;\;\;\;\;\;\;\;\;\;\;\;\;\;\;\;\;\;\;\;\;\;\;\;\;\;\;\;\;\;\;\;
\ge \dfrac{{(2y + cX)X}}{2} - y \sum\nolimits_{n = 2}^N {x_n} \\
\displaystyle\;\;\;\;\;\;\;\;\;\;\;\;\;\;\;\;\;\;\;\;\;\;\;\;\;\;\;\;\;\;\;\;\;\;\;\;\;\;\;\;\;\;\;\;\;\;\;\;\;\;\;\;\;\;\;\;\;\;\;
= \dfrac{{(2y + cX)X}}{2} - y (X-y/c) \\
\displaystyle\;\;\;\;\;\;\;\;\;\;\;\;\;\;\;\;\;\;\;\;\;\;\;\;\;\;\;\;\;\;\;\;\;\;\;\;\;\;\;\;\;\;\;\;\;\;\;\;\;\;\;\;\;\;\;\;\;\;\;
= cX^2/2 +y^2/c .\\
\end{array}
\]
Therefore, \old{we have}
\begin{equation}\label{equa:thm52}
\gamma^0(\mathbf{x}) \ge \dfrac{cX^2/2 +y^2/c}{y^2/(2d)+
(2y+cX)X/2}.
\end{equation}}
{Note that $c$, $d$, $X$, and $y$ are {all} positive. Substituting
$\overline c=c/d$ and $\overline y=cX/y$ {in} (\ref{equa:thm52}), we
{obtain}
\begin{equation}\label{equa:thm53}
{\gamma^0(\mathbf{x}) \geq} \dfrac{cX^2/2 +y^2/c}{y^2/(2d)+
(2y+cX)X/2}= \dfrac{c^2X^2/2 +y^2}{y^2c/(2d)+cXy+c^2X^2/2 } =
{\dfrac{\overline y^2 +2}{\overline c+2\overline y+\overline y^2 }.}
\end{equation}
\old{From (\ref{equa:thm51}) we have that $\overline y \ge 1$, and
\begin{equation}\label{equa:thm533}
 \gamma^0(\mathbf{x}) \ge
\dfrac{\overline y^2 +2}{\overline c+2\overline y+\overline y^2 }.
\end{equation}}}

{We have shown earlier that $X\geq y/c$, so that $\overline y\geq
1$. } {On the interval $\overline y \in [1,\infty)$, the minimum
value of the right hand side of (\ref{equa:thm53}) is attained at
\[
\overline y = \max \left\{ {\frac{{2-\overline  c + \sqrt
{{{\overline c }^2} - 4\overline c  + 12} }}{2},1} \right\}
\buildrel \Delta \over = \phi,
\]
and thus,
\[
\gamma^0(\mathbf{x}) \ge  \dfrac{\phi^2 +2}{\phi^2+ 2 \phi+
\overline c} =f(\overline c).
\]}

\noindent \textbf{The case  $\alpha_1>0$}

{We now consider the case where $\alpha_n > 0$ for every $n$. By
rescaling the cost coefficients and permuting the supplier indices,
we can assume that $\min_n\{ \alpha_n\} = \alpha_1 = 1$.  Suppose
first that $x_1 = 0$. The second inequality in (\ref{equa:nece})
implies that $p(X)\le 1$. Proposition \ref{Prop:trivial} {also}
implies that $X>0$ {so that} there exists some $n$ {for which}
$x_n>0$. The first equality in (\ref{equa:nece}) yields,
$$
\alpha_n=p(X) + x_n p'(X) {\leq p(X)} \le 1.
$$
Since $\alpha_{{n}} \ge 1$, we {obtain} $p(X)=1$  and $ p'(X)=0$.
Then, the vector $\mathbf{x}$ satisfies the optimality conditions in
(\ref{equa:optimality}), and thus is socially optimal in the
original {model}. {But this would contradict}  the fact that
$p(X)\ne p(X^S)$. We conclude that $x_1 > 0$.}

{If $p'(X)$ {were equal to zero,} then the necessary conditions
(\ref{equa:nece}) {would} imply the optimality conditions
(\ref{equa:optimality}), and $\mathbf{x}$ {would be} socially
optimal in the {modified} game. Therefore, we {must} have $p'(X)<0$
{and} $c>0$. The first equality in (\ref{equa:nece}) yields $y >1$,
\(
 x_1 = (y-1)/c \), {and $X\geq (y-1)/c$.  We also} have
\begin{equation}\label{equa:thm54}
0 \le \sum_{n=2}^N x_n =X- \frac{y-1}{c},
\end{equation}
 from which we conclude that $X \ge
(y-1)/c$.}

From Proposition \ref{Prop:piecewise_linear}, the efficiency
$\gamma^0(\mathbf{x})$ {of} the vector $\mathbf{x}$ in the modified
{model} cannot be more than {its efficiency $\gamma(\bx)$} in the
original {model}. {So,}  it suffices to consider the efficiency of
$\mathbf{x}$ in the modified {model}. From the optimality conditions
(\ref{equa:optimality}), we have that $p^0(X^S)=1$, and thus, {using
the definition of $d$,}
$$
X^S=X+ \frac{y-1}{d}.
$$
The optimal social welfare in the modified {model} is
\[
\int_0^{X^S} {p^0(q)\, dq}  - X^S = \dfrac{{{y^2-1}}}{{2d}}+
\dfrac{(2y+cX)X}{2}-X-\dfrac{y-1}{d}=\dfrac{{{{(y - 1)}^2}}}{{2d}} +
X(y - 1) + \dfrac{{c{X^2}}}{2}.
 \]
Note that for $n = 1$ we have $\alpha_n x_n = x_1$. For $n \ge 2$
{and} whenever $x_n > 0$, from the first equality in
(\ref{equa:nece}) we have $\alpha_n=y-x_n c$ and $\alpha_n
x_n=(y-x_n c)x_n$. Hence, in the modified model, the social welfare
associated with $\mathbf{x}$ is
\[
\begin{array}{l}
\displaystyle
\int_0^X {{p^0}(q)\,dq - \sum\nolimits_{n = 1}^N {{\alpha _n}{x_n}} }  = Xy+ c X^2/2 - x_1-\sum\nolimits_{n = 2}^N {(y - {x_n}c){x_n}} \\
\displaystyle\;\;\;\;\;\;\;\;\;\;\;\;\;\;\;\;\;\;\;\;\;\;\;\;\;\;\;\;\;\;\;\;\;\;\;\;\;\;\;\;\;\;\;\;\;\;\;\;\;\;\;\;\;\;\;\;\;\;\;\;
\ge  Xy+ c X^2/2 - x_1- y \sum\nolimits_{n = 2}^N {x_n} \\
\displaystyle\;\;\;\;\;\;\;\;\;\;\;\;\;\;\;\;\;\;\;\;\;\;\;\;\;\;\;\;\;\;\;\;\;\;\;\;\;\;\;\;\;\;\;\;\;\;\;\;\;\;\;\;\;\;\;\;\;\;\;\;
=  X(y-1)+ c X^2/2 - (y-1) \sum\nolimits_{n = 2}^N {x_n} \\
\displaystyle\;\;\;\;\;\;\;\;\;\;\;\;\;\;\;\;\;\;\;\;\;\;\;\;\;\;\;\;\;\;\;\;\;\;\;\;\;\;\;\;\;\;\;\;\;\;\;\;\;\;\;\;\;\;\;\;\;\;\;\;
= X(y-1)+ c X^2/2 - (y-1) (X-(y-1)/c) \\
\displaystyle\;\;\;\;\;\;\;\;\;\;\;\;\;\;\;\;\;\;\;\;\;\;\;\;\;\;\;\;\;\;\;\;\;\;\;\;\;\;\;\;\;\;\;\;\;\;\;\;\;\;\;\;\;\;\;\;\;\;\;\;
= cX^2/2 +(y-1)^2/c .\\
\end{array}
\]
Therefore, \old{we have}
\begin{equation}\label{equa:thm55}
\gamma^0(\mathbf{x}) \ge \dfrac{cX^2/2 +(y-1)^2/c}{(y-1)^2/(2d)+
X(y-1)+cX^2/2}.
\end{equation}

Note that $c$, $d$, $X$, and $y-1$ are {all} positive. Substituting
$\overline c=c/d$ and $\overline y=(cX)/(y-1)$ {in}
(\ref{equa:thm55}), we {obtain}
\begin{equation}\label{equa:thm56}
 \gamma^0(\mathbf{x}) \ge
\dfrac{2\overline y^2 +1}{\overline c\hspace{1pt} \overline
y^2+2\overline y+1}.
\end{equation}
From (\ref{equa:thm54}) we have that $\overline y \ge 1$. On the
interval $\overline y \in [1,\infty)$, the minimum value of the
right hand side of (\ref{equa:thm56}) is attained at
\[
\overline y= \min \left\{ {\frac{{2-\overline c   + \sqrt
{{{\overline c }^2} - 4\overline c  + 12} }}{2},1} \right\}
\buildrel \Delta \over = \phi,
\]
and thus,
\[
\gamma^0(\mathbf{x}) \ge \dfrac{\phi^2 +2}{\phi^2+ 2 \phi+ \overline
c}=f(\overline c).
\]

\pdfbookmark[1]{Proof of Corollary \ref{Coro:mu}}{B3}

\hypertarget{B3}{}

\subsection*{B.3. Proof of Corollary \ref{Coro:mu}}\label{sec:B3}

Let $\bx$ be a Cournot candidate. Since the inverse demand function
is convex, we have that $\mu \ge 1$. If $X>Q$, then \old{we have}
$p(X)=p'(X)=0$. The necessary conditions
\eqref{equa:nece1}-\eqref{eq:necsec} imply the optimality condition
in (\ref{equa:optimality}), and thus $\gamma(\bx)=1>f(\mu)$.

Now consider the case $X \le Q$. If $p(X)=p(X^S)$ for some social
optimum $\bx^S$, then Proposition \ref{Prop:equal} implies that
$\gamma(\bx)=1>f(\mu)$. Otherwise, for any social optimum $\bx^S$,
we have that $\overline c=c/d \le \mu$. Theorem \ref{thm:convex}
shows that the efficiency of every Cournot candidate cannot be less
than $f(\overline c)$. The desired result then follows from the fact
that $f(\overline c)$ is decreasing in $\overline c$.

\hypertarget{B4}{}

\pdfbookmark[1]{Proof of Corollary \ref{Coro:st_lower_bound}}{B4}

\subsection*{B.4. Proof of Corollary
\ref{Coro:st_lower_bound}}\label{sec:B4}

Let $\bx$ and ${\bx}^S$ be a Cournot candidate and  a social
optimum, respectively. If $p(X)=p(X^S)$, for some social optimum
$\bx^S$, then $\gamma(\bx)=1$ and the desired result {holds
trivially.} Now {suppose that} $p(X)\ne p(X^S)$. We first derive an
upper
 bound on the aggregate supply  at a social optimum, and then establish
 a lower bound on the aggregate supply at a Cournot
candidate. The desired results will follow from the fact that the
function $f(\cdot)$ is strictly decreasing.

\emph{Step 1: There exists a social optimum with an aggregate supply
no more than $s$.}

 According to Proposition \ref{Prop:positive}
we have $X^S>0$ {and} there exists a supplier $n$ such that
$x^S_n>0$. From the optimality conditions (\ref{equa:optimality}) we
have $ p(X^S)=C'_n(x^S_n)$, which implies that $p(X^S) \ge C'_n(0)$,
due to the convexity of {the} cost functions. We conclude that
\begin{equation}\label{equa:s_lower_bound0}
p(X^S) =C'_n(x^S_n) \ge C'_n(0) \ge \min_n C'_n(0).
\end{equation}

If $p(X^S) > \min_n C'_n(0)$, then from the definition of $s$ in
(\ref{equa:t}),  and the assumption that $p(\cdot)$ is
nonincreasing, we have that $ X^S < s$.

If $p(X^S) = \min_n C'_n(0)$, by \eqref{equa:s_lower_bound0} we know
that for any $n$ such that $x^S_n>0$, {we must have}
$C'_n(x^S_n)=C'_n(0)=p(X^S)$. Since $C_n(\cdot)$ is convex, we
conclude that $C_n(\cdot)$ is actually linear on the interval
$[0,x^S_n]$. We now argue that there exists a social optimum $
\bx^S$ such that $ X^S \le s$. If $X^S \le s$, then we are done.
Otherwise,  we have $X^S
>s$. Let $\mathcal {N}$ be the set of the ind{ic}es of suppliers who
produce a positive quantity at $\bx^S$. Since $p(\cdot)$ is
nonincreasing, and $p(s)=\min_n C'_n(0)=p(X^S)$, we know that for
any $q \in [s,X^S]$, $p(q)=C'_n(0)$ for every $ n \in \mathcal {N}$.
Combing with the fact that for each supplier $n$ in the set
$\mathcal {N}$, $C_n(\cdot)$ is
 linear on the interval $[0,x^S_n]$, we have
\[
\int_s^{{X^S}} {p(q)\,dq = ({X^S} - s)C'_{n}(x),\qquad \forall n \in
\mathcal {N},\;\;\;\forall x \in [0,x_n^S)},
\]
from which we conclude that the vector, $(s/X^S) \cdot \bx^S$,
yields the same social welfare as $\bx^S$, and thus is socially
optimal. Note that the aggregate supply at $(s/X^S) \cdot \bx^S$ is
$s$.

 If $p(X)=p( X^S)$, then $\gamma(\bx)=1$
and the desired result {holds} trivially. Otherwise, since
$p(\cdot)$ is nonincreasing and convex, we have
\begin{equation}\label{equa:s_lower_bound1}
\left| \partial_- p(s) \right| \le \left| (p(X^S)-p(X))/(X^S-X)
\right| {=d}.
\end{equation}

\emph{Step 2: The aggregate supply at a Cournot candidate $\bx$ is
at least $t$.}

Since $p(\cdot)$ is convex, $\bx$ satisfies the necessary conditions
in (\ref{equa:nece}). {Therefore,}
\begin{equation}\label{equa:t_lower_bound1}
C'_n(x_n) \ge p(X)+ x_n p'(X),\qquad \forall n.
\end{equation}
Since $X \ge x_n$, we have
\begin{equation}\label{equa:t_lower_bound2}
C'_n(x_n) \le C'_n(X), \qquad Xp'(X) \le x_n p'(X),
\end{equation}
where the first inequality follows from the convexity of {the} cost
functions, and the second one is true because $p'(X)\le 0$. Combing
(\ref{equa:t_lower_bound1}) and (\ref{equa:t_lower_bound2}), we have
$$
C'_n(X) \ge p(X)+ X p'(X),\qquad \forall n,
$$
which implies that $X \ge t$. Since $p(\cdot)$ is nonincreasing and
convex, we have
\begin{equation}\label{equa:st_lower_bound1}
{c=\left| p'(X) \right| \leq \left| \partial_+ p(t) \right| .}
\end{equation}
Since $\partial_-p(s)<0$, {Eqs.} (\ref{equa:s_lower_bound1}) and
(\ref{equa:st_lower_bound1}) {yield}
$$
\overline c =c/d \le \partial_+p(t)/
\partial_-p(s).
$$
The desired result \old{then} follows from Theorem \ref{thm:convex},
and the fact that $f(\cdot)$ is strictly decreasing.

\hypertarget{C}{}

\pdfbookmark[0]{Proofs of the results in Section \ref{sec:mono}}{C}

\section*{Appendix C: Proofs of the results in Section \ref{sec:mono}}

\pdfbookmark[1]{Proof of Theorem \ref{thm:mono}}{C1}

\subsection*{C.1. Proof of Theorem
\ref{thm:mono}}

According to the {discussion} in Section \ref{sec:mono}, we only
need to lower bound the efficiency of a Cournot equilibrium $\bx$
{in} a model with $N=1$. Since $N=1$, we can identify the vectors
$\bx$ and $\bx^S$ with the scalars $X$ and $X^S$.
 If $p(X)=p(X^S)$, then according to Proposition \ref{Prop:equal},
the efficiency of the Cournot equilibrium, ${X}$, must equal one,
which establishes part (a).

 We now turn to the proof of
part (b), and we assume that $p(X) \ne
 p(X^S)$. According to Proposition \ref{Prop:sameprice}, we know that $\bx$ cannot be socially optimal.
We will consider separately the cases where $\alpha_1=0$ and
$\alpha_1>0$.

We will again rely on Proposition \ref{Prop:piecewise_linear},
according to which the efficiency of a Cournot candidate $\bx$ is
lower bounded by the efficiency $\gamma^0(\bx)$ of $\bx$ in a model
involving the piecewise linear and convex inverse demand function
$p^0(\cdot)$. Note that since $p(X) \ne
 p(X^S)$, we have that $d>0$.  {As shown in} the proof of Theorem
\ref{thm:convex}, we {have}  $p'(X)<0$, i.e., $c>0$. For
conciseness, we let $y=p(X)$ throughout the proof.

\noindent \textbf{The case $\alpha_1=0$}

Applying conditions \eqref{equa:nece} to the supplier we have
 \( X = y/c. \)
From Proposition \ref{Prop:piecewise_linear}, it suffices to show
that $\gamma^0(\mathbf{x}) \ge 3/(3+\overline c)$. The optimal
social welfare in the modified model is
\begin{equation}\label{equa:optimal_welfare0_mono}
\int_0^{\infty} {p^0(q)\,dq}  - 0 = \int_0^{X+y/d} {p^0(q)\,dq}  - 0
= \dfrac{{{y^2}}}{{2d}}+ \dfrac{(2y+cX)X}{2}.
 \end{equation}
In the modified model, the social welfare associated with
$\mathbf{x}$ is
$$
\int_0^{X^P} {p^0(q)\,dq}  - 0 =
 \dfrac{(2y+cX)X}{2}.
$$
Therefore,
$$
\gamma^0(\mathbf{x}) = \dfrac{(2y+cX)X/2}{y^2/(2d)+ (2y+cX)X/2} =
\dfrac{3}{3+\overline c},
$$
where the last equality is true because  $xc=y$.\\

\noindent \textbf{Tightness}

Consider the model introduced in the proof {of} part (c) of Theorem
\ref{thm:linear}. The inverse demand function is
$p(q)=\max\{1-q,0\}$. The supplier's cost function is identically
zero, i.e., $ C_1(x_1)=0$. The profit maximizing
 output is $ x_1= 1/2 $.
We observe that $\gamma(\bx) = 3/4$. \\

\noindent \textbf{The case  $\alpha_1>0$}

We now consider the case where $\alpha_1 > 0$. By rescaling the cost
coefficients and permuting the supplier indices, we can assume that
$\alpha_1 = 1$. Applying conditions \eqref{equa:nece} to the
supplier,  we obtain
 \( X = (y-1)/c. \)

According to Proposition \ref{Prop:piecewise_linear},  it suffices
to show that the efficiency of $\bx$ in the modified model,
$\gamma^0(\mathbf{x})$, is at least  $3/(3+\overline c)$.
 From the optimality conditions
(\ref{equa:optimality}) we have that $p^0(X^S)=1$, and therefore,
$$
X^S=X+(y-1)/d.
$$
The optimal social welfare achieved in the modified model is
\[
\int_0^{X^S} {p^0(q)\,dq}  - X^S = \dfrac{{{y^2-1}}}{{2d}}+
\dfrac{(2y+cX)X}{2}-X-\dfrac{y-1}{d}=\dfrac{{{{(y - 1)}^2}}}{{2d}} +
X^P(y - 1) + \dfrac{{c{(X)^2}}}{2}.
 \]
In the modified model, the social welfare associated with
$\mathbf{x}$ is
$$
\int_0^{X} {p^0(q)\,dq}  - X =
 X(y - 1) + \dfrac{{c{ {X}^2}}}{2}.
$$

Since \( c X = y-1 \), we have
$$
\gamma^0(\mathbf{x}) = \dfrac{3}{3+\overline c}.
$$

\hypertarget{C2}{}

\pdfbookmark[1]{Proof of Corollary
\ref{Coro:st_lower_bound_mono}}{C2}

\subsection*{C.2. Proof of Corollary
\ref{Coro:st_lower_bound_mono}}\label{sec:C2}

Note that the efficiency of a monopoly output equals the efficiency
of a Cournot equilibrium in a modified model with $N=1$. Therefore,
the desired result follows from the proof of Corollary
\ref{Coro:st_lower_bound}, except that the general lower bound
$f(\cdot)$ is replaced by the tighter one, $3/(3+\overline c)$,
provided by
 Theorem \ref{thm:mono}.

\end{APPENDICES}

\end{document}